\documentclass[reqno,11pt]{amsart}

 \allowdisplaybreaks \numberwithin{equation}{section}

\usepackage[latin1]{inputenc}
\usepackage[T1]{fontenc}
\usepackage{ae}
\usepackage{aecompl}
\usepackage{amsmath}
\usepackage{amsthm}
\usepackage{amsfonts}
\usepackage{amssymb}
\usepackage{mathrsfs}
\usepackage[dvips]{graphicx}
\usepackage{float}
\usepackage[all]{xy}
\usepackage{color}
\theoremstyle{plain}
\newtheorem{prop}{Proposition}[section]

\newtheorem{teo}[prop]{Theorem}

\newtheorem{cor}[prop]{Corollary}
\newtheorem{fact}[prop]{Fact}
\newtheorem{lem}[prop]{Lemma}

\newtheorem{prob}[prop]{Problem}
\theoremstyle{remark}
\newtheorem{oss}[prop]{Remark}
\definecolor{red}{rgb}{1,0,0}
\definecolor{green}{rgb}{0,1,0.2}

\title{Components of moduli spaces of spin curves with the expected codimension II}
\author{Luca Benzo}
\date{}
\begin{document}
\maketitle
\begin{abstract}
\noindent We prove that for all integers $r \geq 2$ and $g \geq \lfloor \frac{r^2+10r+1}{4} \rfloor$ there exists a component of the locus $\mathcal{S}^r_g$ of spin curves with a theta characteristic $L$ such that $h^0(L) \geq r+1$ and $h^0(L)\equiv r+1 (\text{mod } 2)$ which has expected codimension $\binom{r+1}{2}$ inside the moduli space $\mathcal{S}_g$ of spin curves of genus $g$.
\end{abstract}
\tableofcontents
\newpage
\begin{section}{Introduction}
Let $g \geq 1$ be an integer and let $\mathcal{S}_g$ be the moduli space of smooth spin curves parameterizing pairs $(C,L)$ where $C$ is a smooth curve of genus $g$ and $L$ is a theta-characteristic i.e. a line bundle on $C$ such that $L^2 \cong \omega_C$. It is well-known (see \cite{MU}) that $\mathcal{S}_g$ consists of exactly two connected components, $\mathcal{S}^{\text{odd}}_g$ and $\mathcal{S}^{\text{even}}_g$, depending on the parity of $h^{0}(L)$.\\
It is therefore natural to consider, for a fixed integer $r \geq 0$, the sublocus $\mathcal{S}^r_g$ of $\mathcal{S}^{\text{odd}}_g$ or $\mathcal{S}^{\text{even}}_g$ defined by
$$\mathcal{S}^{r}_{g} := \left\{(C,L) \in \mathcal{S}_g : h^{0}(L) \geq r+1 \hbox{ and } h^{0}(L) \equiv r+1 (\text{mod } 2)\right\}.$$
In \cite{HA} Harris proved that each irreducible component of $\mathcal{S}^r_g$ has dimension $\geq 3g-3- \binom{r+1}{2}$. Since the map $\pi:\mathcal{S}_g \rightarrow \mathcal{M}_g$ sending a point $(C,L)$ to $[C]$ is finite, this is equivalent to say that each irreducible component of $\mathcal{S}^r_g$ has codimension less than or equal to $\binom{r+1}{2}$ in $\mathcal{S}_g$. If the codimension in $\mathcal{S}_g$ of a component $V_{r,g} \subset \mathcal{S}^r_g$ equals $\binom{r+1}{2}$, we also say that $V_{r,g}$ has the \emph{expected codimension (in $\mathcal{S}_g$)}. It is natural to pose the following \emph{geography problem}:
\begin{prob}
For which pairs of positive integers $(r,g)$ there exists an irreducible component of $\mathcal{S}^r_g$ having expected codimension?
\end{prob}
An obvious necessary condition for the existence of such a component is that the inequality $3g-3- \binom{r+1}{2} \geq 0$ must be satisfied, i.e. points $(r,g)$ for which there exists a component of $\mathcal{S}^r_g$ with the expected codimension must lie above the parabola $g=\frac{r(r+1)}{6}+1$.\\
Inside this region of the plane $r,g$, many results are known for small $r$ (cf. \cite{BEA}, \cite{TEI} and \cite{FA}).
Moreover, in our previous paper \cite{BE} we were able to prove the following
\begin{teo}
\label{teobenzo4}
For all integers $r \geq 2$ and $g \geq {r+2 \choose 2}$, the locus $\mathcal{S}^r_{g}$ has an irreducible component of codimension ${r+1 \choose 2}$ in $\mathcal{S}_{g}$.
\end{teo}
thus ``covering" the plane region above the parabola $g={r+2 \choose 2}=\frac{(r+2)(r+1)}{2}$.
The purpose of the present paper is to prove a much stronger result, namely the following
\begin{teo}
\label{fundteo}
For all integers $r \geq 2$ and $g \geq g(r) := \lfloor \frac{r^2+10r+1}{4} \rfloor$, the locus $\mathcal{S}^r_{g}$ has an irreducible component of codimension ${r+1 \choose 2}$ in $\mathcal{S}_{g}$.
\end{teo}
The way Theorem \ref{fundteo} is proved also provides an existence result for irreducible components of $\mathcal{S}^r_g$ parameterizing pairs $(C,L)$ such that $L$ is \emph{very ample} (and in particular $C$ is not hyperelliptic).\\
%The existence of such irreducible components is a highly non-trivial fact.
In order to prove Theorem \ref{fundteo}, we will use the following fundamental result due to Farkas, which is proved by inductively smoothing a stable spin structure (see \cite{COR} for details) over a suitable reducible curve.
\begin{prop}[{\cite[Proposition 2.4 and Remark 2.5]{FA}}]
\label{farkasinductive}
Fix integers $r,g_0 \geq 1$. If $\mathcal{S}^r_{g_0}$ has an irreducible component $V_{r,g_0}$ of codimension $\binom{r+1}{2}$ in $\mathcal{S}_{g_0}$, then for every $g \geq g_0$ the locus $\mathcal{S}^r_g$ has an irreducible component $V_{r,g}$ of codimension $\binom{r+1}{2}$ in $\mathcal{S}_g$. Moreover, if the general point $(C,L) \in V_{r,g_0}$ has very ample $L$, then the general point $(C^{\prime},L^{\prime}) \in V_{r,g}$ has very ample $L^{\prime}$.
\end{prop}
The core of the proof of our result is the construction of a reducible curve $X \subset \mathbb{P}^r$ which is the central fibre of a (flat) family $\mathcal{X} \rightarrow B$, $\mathcal{X} \subset \mathbb{P}^r \times B$, where $B$ is an irreducible algebraic scheme, such that the general fibre $\mathcal{X}_b$ is a smooth half-canonical curve (i.e. $\mathcal{O}_{\mathcal{X}_b}(1)$ is a theta-characteristic) having nice properties ensuring expected codimension of the component of $\mathcal{S}^r_g$ parameterizing it. Namely, we will require the \emph{Gaussian map} of the line bundle embedding $\mathcal{X}_b$ in $\mathbb{P}^r$ to be injective (see Subsection \ref{gaussianmap} for a precise definition).\\
The idea is, as in \cite{BE}, to build $X$ in such a way that the line bundle $\mathcal{O}_X(2)$ has degree $2p_a(X)-2$ and $h^1(\mathcal{O}_X(2))=1$. If the cohomology numbers of $\mathcal{O}_X(2)$ remain constant along the fibres of the family $\mathcal{X} \rightarrow B$ (a property that will follow from the $2$-normality of $X$), then $\mathcal{O}_{\mathcal{X}_b}(2)$ will be isomorphic to the (unique) line bundle of degree $2g(\mathcal{X}_b)-2$ and speciality index 1 on $\mathcal{X}_b$, the canonical one.\\
Assume for the sake of simplicity that $X := C \cup D$ is the nodal union of two irreducible components and that these components are 2-normal. Let $\Delta$ be the intersection scheme of $C$ and $D$. Note that the cohomology sequence associated to
\begin{equation}
\label{ix2opr2ox2}
0 \rightarrow \mathcal{I}_X(2) \rightarrow \mathcal{O}_{\mathbb{P}^r}(2) \rightarrow \mathcal{O}_X(2) \rightarrow 0
\end{equation}
gives immediately $h^1(\mathcal{O}_X(2))=h^2(\mathcal{I}_X(2))$ (an analogous property holds for $C$ and $D$), hence imposing the condition $h^1(\mathcal{O}_X(2))=1$ leaves only two possibilities:
%O i punti dello schema di intersezione di C e D impongono condizioni indipendenti alle quadriche di P^r e
\begin{itemize}
\item[(a)] \quad $0 \rightarrow \mathcal{I}_{X}(2) \rightarrow \underbrace{\mathcal{I}_C(2)}_{h^2=h^1(\mathcal{O}_C(2))=1} \oplus \quad \underbrace{\mathcal{I}_D(2)}_{h^2=0} \quad \rightarrow \underbrace{\mathcal{I}_{\Delta}(2)}_{h^1=0, \, h^2=0}  \rightarrow 0$
\end{itemize}
(the points of $\Delta$ impose independent conditions to hyperquadrics in $\mathbb{P}^r$ and one between $\mathcal{O}_C(2)$ and $\mathcal{O}_D(2)$, say the former, has speciality index 1)
%Oppure i punti di \Delta impongono una condizione in meno alle quadriche
\begin{itemize}
\item[(b)] \quad $0 \rightarrow \mathcal{I}_X(2) \rightarrow \quad \underbrace{\mathcal{I}_C(2)}_{h^2=0} \quad \oplus \quad \underbrace{\mathcal{I}_D(2)}_{h^2=0} \quad \rightarrow \underbrace{\mathcal{I}_{\Delta}(2)}_{h^1=1, \, h^2=0}  \rightarrow 0$
\end{itemize}
(the points of $\Delta$ impose one condition less to hyperquadrics in $\mathbb{P}^r$ and both $\mathcal{O}_C(2)$ and $\mathcal{O}_D(2)$ are nonspecial).\\
The construction carried out in the present paper is a far-reaching generalisation of the one in \cite{BE} and both are an incarnation of case (a). The (above remarked) fact that property $h^1(\mathcal{O}_C(2))=1$ holds for half-canonical curves clearly suggests the possibility to proceed by induction on $r$ (it would be interesting to understand whether it is possible to sharpen the result in Theorem \ref{fundteo} by exploiting an incarnation of case (b)).\\
On the other hand, extending Theorem \ref{teobenzo4} to Theorem \ref{fundteo} presents a series of interesting technical problems which require a significant amount of additional work, and whose solution constitutes one of the main points of interest of this paper, since it is in principle generalisable to many other situations dealing with degenerations of projective curves.\\
In particular, in the ``new" region of the plane $r,g$ that must be covered to prove Theorem \ref{fundteo}, the numerology of our inductive construction gives $h^1(N_{X/\mathbb{P}^r})>>0$ for $r >>0$, hence standard techniques fail to provide and argument for the smoothability of $X$ in $\mathbb{P}^r$ and have to be replaced by a delicate verification of the assumptions of our Lemma \ref{critical}, namely a suitable cohomological condition as well as the smoothness of the Hilbert scheme at the point $[X]$.\\
Techniques involving elementary transformations of vector bundles and varieties of secant divisors will play a major role as well.\\ \\

We work over the field of complex numbers. If $Y$ is an algebraic scheme and $X \subset Y$ is a closed subscheme with defining ideal sheaf $\mathcal{I}_{X/Y}$, we will denote by $N^{\vee}_{X/Y} := \mathcal{I}_{X/Y}/\mathcal{I}^2_{X/Y}$ the \emph{conormal sheaf of $X$ in $Y$} and by $N_{X/Y}$ its dual $\textbf{Hom}_{\mathcal{O}_X}(\mathcal{I}_{X/Y}/\mathcal{I}^2_{X/Y},\mathcal{O}_X)$, which we will call the \emph{normal sheaf of $X$ in $Y$}. If there is no ambiguity, we will possibly write $N^{\vee}_{X}$ ($N_X$) instead of $N^{\vee}_{X/Y}$ ($N_{X/Y}$). The same will happen with $\mathcal{I}_{X/Y}$ and $\mathcal{I}^2_{X/Y}$.\\
By a \emph{family} of curves and a \emph{deformation} of a curve we will always mean flat ones.\\
We will write $\text{Hilb}^r_{g,d}$ for the Hilbert scheme of curves of arithmetic genus $g$ and degree $d$ in $\mathbb{P}^r$.
\end{section}
\begin{section}{Preliminary results}
\begin{subsection}{Generalities about theta characteristics}
Theta characteristics were classically studied from a trascendental point of view for their connection with theta functions. Mumford presented in \cite{MU} a completely algebraic viewpoint of the subject, which was later carried on by Harris.
\begin{teo}[{\cite[Theorem]{MU}}, {\cite[Theorem 1.10]{HA}}]
\label{hm}
Let $\mathcal{S} \rightarrow B$ be a flat family over an irreducible algebraic scheme $B$ such that, for all $\lambda \in B$, the fibre $\mathcal{S}_{\lambda}$ is a Gorenstein reduced curve, and let $\mathcal{L} \in \emph{Pic}(\mathcal{S})$ such that $\mathcal{L}_{\lambda}^{2} \cong \omega_{\mathcal{S}_{\lambda}}$ for all $\lambda \in B$. Then:
\begin{itemize}
\item[$(i)$] the map $\rho: B \rightarrow \mathbb{N}_{0}$, $\lambda \mapsto h^{0}(\mathcal{S}_{\lambda},\mathcal{L}_{\lambda})$ is constant modulo 2;
\item[$(ii)$] the locus $B_r := \left\{\lambda \in B : \rho(\lambda)=r+1\right\}$ is empty or has codimension less than or equal to ${r+1 \choose 2}$ in $B$.
\end{itemize}
\end{teo}
Property $(i)$ gives full motivation for the study of the irreducible components of the loci
$\mathcal{S}^{r}_{g} := \left\{(C,L) \in \mathcal{S}_g : h^{0}(L) \geq r+1 \hbox{ and } h^{0}(L) \equiv r+1 (\text{mod } 2)\right\}$
and $\mathcal{M}^{r}_g := \pi(\mathcal{S}^r_g)$, where $\pi:\mathcal{S}_g \rightarrow \mathcal{M}_g$ is the map sending the point $(C,L)$ to $[C]$.
The fundamental fact we are interested in is the following
\begin{prop}
\label{corharris}
$\mathcal{S}^r_g$ is empty if and only if $r > \frac{g-1}{2}$. If $r \leq \frac{g-1}{2}$, for every irreducible component $V \subset \mathcal{S}^r_g$ one has
\begin{equation}
\label{dimv}
\dim V \geq 3g-3-{r+1 \choose 2}.
\end{equation}
\end{prop}
\begin{proof}
If $r > \frac{g-1}{2}$, then $\mathcal{S}^r_g$ is empty by Clifford's theorem on special divisors on smooth curves. On the other hand, the locus $\mathcal{M}^r_g$, $r \leq {\lfloor\frac{g-1}{2}\rfloor}$ is not empty, as it contains the hyperelliptic locus. The second assertion immediately follows from $(ii)$ of Theorem \ref{hm}.
\end{proof}
If equality holds in (\ref{dimv}), we say that $V$ has \emph{expected codimension (in $\mathcal{S}_g$)}.
\end{subsection}
\begin{subsection}{The Gaussian map}
\label{gaussianmap}
Let $X$ be a smooth projective variety and $L \in \text{Pic}(X)$. Let $R(L)$ be the kernel of the multiplication map $H^{0}(L) \otimes H^{0}(L) \rightarrow H^{0}(L^2)$. One has $R(L)=\wedge^2 H^0(L) \oplus I_2(L)$, where $I_2(L) := \ker \left\{ \text{Sym}^2(L) \rightarrow H^{0}(L^2) \right\}$ is the kernel of the restriction of the multiplication map to the 2-symmetric part of the tensor product. Consider the map
$$\Phi_L: R(L) \rightarrow H^{0}(\Omega^1_X \otimes L^2)$$
defined, locally on a Zariski open subset of $X$ over which $L$ is trivial,  by the correspondence
$$s \otimes t \mapsto s\,dt-t\,ds.$$
Clearly, $\Phi_L$ vanishes on symmetric tensors and one can consider the restriction
$$\Psi_L := {\Phi_L}_{|\wedge^2 H^0(L)}:\wedge^2 H^0(L) \rightarrow H^0(\Omega^1_X \otimes L^2)$$
which is called the \emph{Gaussian map of} $L$.\\
Let now $L$ be very ample and identify $X$ with its image in $\mathbb{P}^r$ via the embedding given by $|L|$. Dualizing the Euler sequence
$$0 \rightarrow \mathcal{O}_X \rightarrow H^{0}(\mathcal{O}_X(1))^{\vee} \otimes \mathcal{O}_X(1) \rightarrow {T_{\mathbb{P}^r}}_{|X} \rightarrow 0$$
and twisting by $\mathcal{O}_X(2)$, one obtains the exact sequence
$$0 \rightarrow {\Omega^1_{\mathbb{P}^r}}_{|X}(2) \rightarrow H^{0}(\mathcal{O}_X(1)) \otimes \mathcal{O}_X(1) \rightarrow \mathcal{O}_X(2) \rightarrow 0$$
whose associated cohomology sequence yields $R(L)=H^{0}({\Omega^1_{\mathbb{P}^r}}_{|X}(2))$. Moreover, $\Phi_L$ is exactly the map associated in cohomology to the map $\alpha$ in the conormal sequence twisted by $\mathcal{O}_X(2)$
$$0 \rightarrow N^{\vee}_X(2) \rightarrow {\Omega^1_{\mathbb{P}^r}}_{|X}(2) \xrightarrow{\alpha} \Omega^1_X(2) \rightarrow 0.$$
and thus one has
\begin{equation}
\label{h0nveex2}
H^{0}(N^{\vee}_X(2)) = \ker \Phi_L=\ker \Psi_L \oplus I_2(L)
\end{equation}
where, as $L$ is very ample, $I_2(L) \cong H^{0}(\mathcal{I}_X(2))$. In particular, $\Psi_L$ is injective if and only if $H^{0}(\mathcal{I}_X(2)) \cong H^{0}(N^{\vee}_X(2))$. By the definition of $N^{\vee}_X$ one has the short exact sequence
\begin{equation}
\label{eq1}
0 \rightarrow \mathcal{I}^2_X(2) \rightarrow \mathcal{I}_X(2) \rightarrow N^{\vee}_X(2) \rightarrow 0.
\end{equation}
Since $X$ is not contained in any hyperplane of $\mathbb{P}^r$, one has $H^{0}(\mathcal{I}^2_X(2))=(0)$ (the singular locus of a hyperquadric is empty or a linear space), hence the cohomology sequence of (\ref{eq1}) immediately gives:
\begin{prop}
\label{h1i2c}
Notation as above, if $H^1(\mathcal{I}^2_X(2))=(0)$, then $\Psi_L$ is injective. If moreover $X$ is 2-normal, then the converse holds.
%$\ker \Psi_L \cong H^{1}(\mathcal{I}^2_X(2))$.
\end{prop}
The map $\Psi_L$ has a very interesting relation with spin curves. For a point $(C,L) \in \mathcal{S}^r_g$, the forgetful map $\pi:\mathcal{S}_g \rightarrow \mathcal{M}_g$ gives a natural identification $T_{[C,L]}\mathcal{S}_g = T_{[C]}\mathcal{M}_g=H^{0}(\omega^2_C)^{\vee}$ between the Zariski tangent spaces to the moduli stacks. Nagaraj (\cite{NAG}) has shown that
\begin{teo}[{\cite[Theorem 1]{NAG}}]
Let $(C,L) \in \mathcal{S}^r_g$ and let $\Psi_L:\wedge^2 H^{0}(L) \rightarrow H^{0}(\omega^2_C)$ be the Gaussian map. Then
$$T_{(C,L)}\mathcal{S}^r_g \cong \left(\emph{im}\, \left(\Psi_L  \right)^{\bot}\right)^{\vee} \subset T_{(C,L)}\mathcal{S}_g.$$
In particular, if $\Psi_L$ is injective, then $\dim T_{(C,L)}\mathcal{S}^r_g =3g-3-{r+1 \choose 2}$.
\end{teo}
Combining this result with Proposition \ref{corharris}, one immediately obtains the following
\begin{cor}
\label{cornag}
Let $V \subset \mathcal{S}^r_g$ be an irreducible component and let $(C,L) \in V$ such that $\Psi_L$ is injective. Then
%$(C,L)$ is a smooth point of $V$,
$h^0(L)=r+1$ and $V$ has expected codimension ${r+1 \choose 2}$ in $\mathcal{S}_g$.
\end{cor}
Let $C \subset \mathbb{P}^r$ be a (nondegenerate) smooth curve of genus $g$. We say that $C$ is \emph{half-canonical} if $\mathcal{O}_C(1)$ is a theta-characteristic on $C$. Proposition \ref{h1i2c} then gives
\begin{fact}
\label{ineq}
Let $C \subset \mathbb{P}^r$ be a smooth linearly normal (nondegenerate) half-canonical curve of genus $g$. Assume that $H^1(\mathcal{I}^2_C(2))=(0)$. Then the pair $\left(C,\mathcal{O}_C(1)\right)$ is parameterized by an irreducible component of $\mathcal{S}^r_g$ having the expected codimension in $\mathcal{S}_g$.
%Ma lo stack è liscio in quel punto, quindi la varietà è localmente irriducibile in quel punto.
\end{fact}
which will be our fundamental criterion to check expected codimension in the sequel.\\
Another important fact concerns Hilbert schemes:
\begin{prop}
\label{hilbertscheme}
Let $C \subset \mathbb{P}^r$ be a smooth (nondegenerate) linearly and 2-normal half-canonical curve of genus $g$ such that the Gaussian map $\Psi_{\mathcal{O}_C(1)}$
is injective.
%(C,\mathcal{O}_C(1))$ is a smooth point of an irreducible component $V \subset \mathcal{S}^r_g$ having expected codimension.
Then $[C]$ is a smooth point of $\emph{Hilb}^r_{g,g-1}$. In particular, $\dim_{[C]}\emph{Hilb}^r_{g,g-1}=h^0(N_C)=3g-4+{r+2 \choose 2}$. Moreover, every smooth curve parameterized by the irreducible component $W \subset \emph{Hilb}^r_{g,g-1}$ containing $[C]$ is half-canonical.
%In particular, the half-canonical locus is not of positive codimension in $W$.
\end{prop}
\begin{proof}
Let $\phi:\text{Hilb}^r_{g,g-1} \dashrightarrow \mathcal{M}_g$ be the moduli map.
%We first show that there exists an irreducible component $W \subset \text{Hilb}^r_{g,g-1}$ containing the point $[C]$ such that $\phi_{|W}$ dominates the irreducible component of $\mathcal{M}^r_g$ containing $[C]$. Vero, ma non mi serve dichiararlo.
%[C] è un punto liscio (stack) di S^r_g perché la mappa di Gauss è iniettiva. Liscio stack implica che la varietà è localmente irriducibile in [C].
Let
$$\xymatrix{\mathcal{O}_C(1) \ar[d] \ar[r] & \mathscr{L} \ar[d] \\ C \ar[d] \ar@{^{(}->}[r] & \mathscr{C} \ar[d] \\ [C] \ar@{^{(}->}[r] & B}$$
be the versal deformation of the spin curve $(C,\mathcal{O}_C(1))$. In particular, $\mathscr{L} \in \text{Pic}(\mathscr{C})$ and $\mathscr{L}^2 \cong \omega_{\mathscr{C}/B}$. Let $B_r := \left\{b \in B | h^0(\mathscr{C}_b,\mathscr{L}_b) \geq r+1, \, h^0(\mathscr{C}_b,\mathscr{L}_b) \equiv r+1 (\text{mod } 2)\right\} \subset B$. Since $\mathcal{O}_C(1)$ is very ample and $h^0(\mathcal{O}_C(1))=r+1$, there exists a nonempty open subset $U$ of an irreducible component of $B_r$ containing the point $[C]$ such that $\mathscr{L}_b$ is very ample and $h^0(\mathscr{C}_b,\mathscr{L}_b) = r+1$ for every $b \in U$. This implies that there exists an irreducible component $W \subset \text{Hilb}^r_{g,g-1}$ containing $[\mathscr{C}_b]$, $\mathscr{C}_b \hookrightarrow_{\varphi_{|\mathscr{L}_b|}} \mathbb{P}^r$, for every $b \in U$. By Theorem \ref{hm} $(ii)$, it follows that $\dim \phi(W) \geq 3g-3-{r+1 \choose 2}$.\\
By assumptions and Proposition \ref{h1i2c} one has
$$h^0(N_C)=\chi(N_C)+h^1(N_C)=\chi(N_C)+h^0(N^{\vee}_C(2))=\chi(N_C)+h^{0}(\mathcal{I}_C(2))=$$
$$=\chi(N_C)+h^{0}(\mathcal{O}_{\mathbb{P}^r}(2))-h^0(\mathcal{O}_C(2))=3g-4+{r+2 \choose 2}=3g-3-{r+1 \choose 2}+(r+1)^2-1.$$
Let $C^{\prime}$ be any deformation of $C$ in $\mathbb{P}^r$. Since $\chi(\mathcal{I}_{C^{\prime}}(2))=\chi(\mathcal{I}_C(2))$ and $h^1(\mathcal{I}_C(2))=0$, sequence (\ref{ix2opr2ox2}) for $C$ and $C^{\prime}$ and the upper semicontinuity of the cohomology give $1=h^2(\mathcal{I}_{C}(2))=h^2(\mathcal{I}_{C^{\prime}}(2))=h^{1}(\mathcal{O}_{C^{\prime}}(2))$. Then, using again the upper semicontinuity of the cohomology, one has, for any smooth curve $C^{\prime \prime}$ parameterized by $W$, $h^1(\mathcal{O}_{C^{\prime \prime}}(2)) \geq 1$, and, since $\deg \mathcal{O}_{C^{\prime \prime}}(2)=2g-2$, one must have equality. This immediately implies $\mathcal{O}_{C^{\prime \prime}}(2) \cong \omega_{C^{\prime \prime}}$ i.e. the fact that $C^{\prime \prime}$ is half-canonical. Since the number of theta characteristics on a fixed curve is finite, the linear series embedding $C$ is isolated in the Brill-Noether variety $\mathcal{W}^r_{g-1}(C)$, hence the fibre of $\phi$ at $[C]$
%is smooth with tangent space at $[C]$ isomorphic to $H^{0}(T_{\mathbb{P}^r})$. Vero se W^r_{g-1}(C) è ridotta, che è plausibile.
has dimension $H^{0}(T_{\mathbb{P}^r})$. One then has that
%Ogni deformatione di C è ancora immersa da una theta, quindi in particolare le curve isomorfe a C e deformazioni di C lo sono.
$\dim_{[C]} \text{Hilb}^r_{g,g-1} = \dim \phi(W)+h^{0}(T_{\mathbb{P}^r}) \geq 3g-3-{r+1 \choose 2}+(r+1)^2-1$. Since $H^{0}(N_C) \cong T_{[C]} \text{Hilb}^r_{g,g-1}$, it follows that $[C]$ is a smooth point of $\text{Hilb}^r_{g,g-1}$.
%e da questo si deduce anche che $W$ è "the" irreducible component containing $[C]$.
%Moreover, the last inequality must be an equality, thus $\dim \phi(W)=3g-3-{r+1 \choose 2}$ and the last part of the claim immediately follows.
\end{proof}
In the end, we mention the fact that will give the base case for the proof of the main theorem.
\begin{prop}
\label{r=2}
The locus $\mathcal{S}^2_{6}$ has an irreducible component $V$ of codimension $3$ in $\mathcal{S}_6$ such that, for a general point $(C,L) \in V$, the line bundle $L$ is very ample and the Gaussian map $\Psi_L$ is injective.
\end{prop}
\begin{proof}
Let $C$ be a smooth plane quintic and $L=\mathcal{O}_C(1)$. The adjunction formula gives that $L$ is a theta-characteristic and $h^{1}(N_C)=0$. By Serre duality and (\ref{h0nveex2}), the Gaussian map $\Psi_L$ is injective, and, by Corollary \ref{cornag}, $V$ has expected codimension in $\mathcal{S}^2_6$.
\end{proof}
\end{subsection}
\begin{subsection}{Deformations of embedded curves}
\label{dec}
In order to proceed in the described direction, we have to briefly recall some basic facts about deformations of embedded curves (see \cite[Section 1]{Sernesi} for an extensive discussion of all the topic), and then state two general lemmas.\\ \\
Let $X \subset \mathbb{P}^r$, $r \geq 2$, be a connected reduced curve with at most nodes as singularities. There is a 4-terms exact sequence
$$0 \rightarrow T_X \rightarrow {T_{\mathbb{P}^r}}_{|X} \rightarrow N_X \rightarrow T^1_X \rightarrow 0$$
where $T_X := \textbf{Hom}(\Omega^1_X,\mathcal{O}_X)$ and $T^1_X$ is the \emph{first cotangent sheaf} of $X$ (see \cite[Section 1.1.3]{SE}), which is a torsion sheaf supported on the singular locus $\text{Sing}(X)$ and having stalk $\mathbb{C}$ at each of the points (if $X$ is smooth, it is obviously the zero sheaf).
For every subset $I \subseteq \text{Sing}(X)$, we will denote by ${T^1_{X}}_{|I}$ the restriction to $I$ of $T^1_X$ extended by zero on $X$, and by $N^{I}_X := \ker \left\{N_X \rightarrow {T^1_X}_{|I} \right\}$.\\
The vector space $H^0(N^{I}_X)$ is isomorphic to the tangent space at $[X]$ to the locally closed subscheme of $\text{Hilb}^r_{p_a(X),\deg(X)}$ parameterizing deformations of $X$ which are locally trivial at the points of $I$ i.e preserve the nodes corresponding to the points of $I$ (if $I=\text{Sing}(X)$ these are called \emph{locally trivial deformations of} $X$). In the case $I=\emptyset$, one recovers the well-known fact that $H^0(N_X) \cong T_{[X]}\text{Hilb}^r_{p_a(X),\deg(X)}$.\\
The sheaf $N'_X := N^{\text{Sing}(X)}_X = \ker \left\{N_X \rightarrow T^1_X \right\}$ is called the \emph{equisingular normal sheaf} of $X$.\\
%For each $p \in S$ let $T^1_p$ denote the restriction to $p$ of $T^1_X$ extended by zero on $X$.\\
From the deformation-theoretic interpretation of the cohomology spaces associated to $N_X$ and $T^1_X$, it follows that, if the cohomology maps $H^{0}(N_X) \rightarrow H^{0}({T^1_X}_{|p})$ are surjective for every $p \in \text{Sing}(X)$ and $\text{Hilb}^r_{p_a(X),\deg(X)}$ is smooth at $[X]$ (a condition that is implied, in particular, by the vanishing of $h^{1}(N_X))$, then the curve $X$ (flatly) deforms to a smooth curve in $\mathbb{P}^r$ i.e. it is a fibre of a family of curves $\mathcal{X} \rightarrow B$, $\mathcal{X} \subset \mathbb{P}^r \times B$, where $B$ is an irreducible algebraic scheme, such that the general fibre $\mathcal{X}_b$ of $\mathcal{X}$ is smooth. $X$ is said to be \emph{smoothable in $\mathbb{P}^r$} (in the sequel we will sometimes omit the specification of the ambient space and simply say \emph{smoothable}). Note that, in particular, $X$ is smoothable if $H^{1}(N'_X)=(0)$.\\
\begin{lem}[cf. {\cite[Lemma 5.1]{Sernesi}}]
\label{sequences}
Let $C_1 \subset \mathbb{P}^r$ be a connected reduced curve with at most nodes as singularities and $C_2 \subset \mathbb{P}^r$ be a smooth connected curve such that $X := C_1 \cup C_2$ is nodal and $\Delta := C_1 \cap C_2$ is a smooth $0$-dimensional subscheme of $C_1$ and $C_2$ supported at smooth points of $C_1$.
Then there exist exact sequences:
\begin{equation}
\label{elementarytransformationnci}
0 \rightarrow N_{C_i} \rightarrow {N_{X}}_{|C_i} \rightarrow {T^1_X}_{|\Delta} \rightarrow 0, \quad i=1,2
\end{equation}
\begin{equation}
\label{icnxn}
0 \rightarrow \mathcal{I}_{C_i/X} \otimes N_{X} \rightarrow N_X \rightarrow {N_{X}}_{|C_i} \rightarrow 0, \quad i=1,2
\end{equation}
\begin{equation}
\label{icinxnprimex}
0 \rightarrow \mathcal{I}_{C_i/X} \otimes N_{X} \rightarrow N^{\Delta}_X \rightarrow N_{C_i} \rightarrow 0, \quad i=1,2
\end{equation}
%\newline
where $\mathcal{I}_{C_1/X} \cong \mathcal{O}_{C_2}(-\Delta)$ and $\mathcal{I}_{C_2/X} \cong \mathcal{O}_{C_1}(-\Delta)$.
\end{lem}
\begin{proof}
By the definition of $N^{\vee}_{C_i/X}$ there exists an exact sequence
\begin{equation}
\label{i2cixicix}
0 \rightarrow \mathcal{I}^2_{C_i/X} \rightarrow \mathcal{I}_{C_i/X} \rightarrow N^{\vee}_{C_i/X} \rightarrow 0
\end{equation}
where $\mathcal{I}^2_{C_1/X} \cong \mathcal{O}_{C_2}(-2\Delta)$ and $\mathcal{I}^2_{C_2/X} \cong \mathcal{O}_{C_1}(-2\Delta)$. One then has $N^{\vee}_{C_i/X} \cong \mathcal{O}_{\Delta}$. By \cite[(D.2) and the proof of Lemma D.1.3 $(ii)$]{SE}, there exists a short exact sequence
\begin{equation}
\label{nveexci}
0 \rightarrow {N^{\vee}_{X}}_{|C_i} \rightarrow N^{\vee}_{C_i} \rightarrow N^{\vee}_{C_i/X} \rightarrow 0.
\end{equation}
Dualizing (\ref{nveexci}) one obtains the short exact sequence
$$0 \rightarrow N_{C_i} \rightarrow {N_{X}}_{|C_i} \rightarrow \textbf{Ext}^1_{\mathcal{O}_{C_i}}\left(\mathcal{O}_{\Delta},\mathcal{O}_{C_i}\right) \rightarrow 0.$$
Since $\textbf{Ext}^1_{\mathcal{O}_{C_i}}\left(\mathcal{O}_{\Delta},\mathcal{O}_{C_i}\right) \cong {T^1_{X}}_{|\Delta}$, the existence of (\ref{elementarytransformationnci}) is proved.\\
Sequence (\ref{icnxn}) is obtained by tensorizing the exact sequence $0 \rightarrow \mathcal{I}_{C_i/X} \rightarrow \mathcal{O}_X \rightarrow \mathcal{O}_{C_i} \rightarrow 0$ by the locally free sheaf $N_X$.\\
From the definition of $N^{\Delta}_X$, there exists an inclusion $\alpha: \mathcal{I}_{C_i/X} \otimes N_{X} \rightarrow N^{\Delta}_X$ fitting into the commutative exact diagram
%tenere fissa una qualsiasi delle due componenti significa in particolare tenere fisso \Delta.
$$\xymatrix{& 0 \ar[d] & 0 \ar[d] &  & \\ 0 \ar[r] & \mathcal{I}_{C_i/X} \otimes N_{X} \ar[r]^{\alpha} \ar[d]^{\beta} & N^{\Delta}_X \ar[r] \ar[d] & \mathcal{A} \ar[r] \ar[d]^{\gamma} & 0\\0 \ar[r] & \mathcal{I}_{C_i/X} \otimes N_{X} \ar[d] \ar[r] & N_X \ar[d] \ar[r] & {N_{X}}_{|C_i} \ar[d]^{\epsilon} \ar[r] & 0 \\ & 0 \ar[r] & {T^1_{X}}_{|\Delta} \ar[r]^{\eta} & \mathcal{B}}$$
where $\mathcal{A} := \text{coker } \alpha$ and $\mathcal{B} := \text{coker } \gamma$. Since $\beta$ is an isomorphism, $\gamma$ must be injective. By the injectivity of $\eta$ and the commutativity of the diagram, one has that $\text{im}\, \epsilon \cong {T^1_{X}}_{|\Delta}$, hence the right vertical sequence is nothing but sequence (\ref{elementarytransformationnci}), $\mathcal{A} \cong N_{C_i}$, and the existence of sequence (\ref{icinxnprimex}) is proved.
%siccome N_X \rightarrow {T^1_{X}}_{|\Delta} è suriettiva e N_X \rightarrow {N_X}_{|C_i} è suriettiva, allora im \eta = im ( {N_X}_{|C_i} \rightarrow %\mathcal{B}
%comparing the right vertical sequence with sequence (\ref{elementarytransformationnci}), one has that $\mathcal{A}$ must be a subsheaf of $N_{C_i}$. A count on Euler characteristics on the upper horizontal sequence then immediately shows that it must be $\mathcal{A} \cong N_{C_i}$, hence the existence of sequence (\ref{njx}) is proved.\\
\end{proof}
The following Lemma is a strong version of Lemma 2.5 in \cite{BE}.
\begin{lem}
\label{critical}
Let $r \geq 2$ be an integer, let $X=\mathcal{X}_0 \subset \mathbb{P}^r$ be a nodal curve and let $q: \mathcal{X} \subset \mathbb{P}^r \times B \rightarrow B$ be a (flat) family of locally trivial deformations of $X$ over an irreducible algebraic scheme $B$. Assume that the \emph{critical scheme} $\mathcal{S} \subset \mathcal{X}$ of the map $q$, i.e. the subscheme of nodal points of the fibres of $q$, is irreducible. Moreover, assume that $h^{0}(N'_X) < h^0(N_X)$ and that the Hilbert scheme $\emph{Hilb}^r_{p_a(X),\deg(X)}$ is smooth at the point $[X]$. Then the curve $X$ is smoothable in $\mathbb{P}^r$.
%la versione predecendente del Lemma che aveva "reduced at the point [X]" al posto di "smooth", è sbagliata. Infatti la proprietà di first order smoothability non è necessariamente aperta, per cui si potrebbe avere un punto singolare dello schema di Hilbert che parametrizza una curva first order smoothable, ma il cui punto generale e liscio in un intorno parametrizza una curva non first order smoothable e quindi non smoothable.
\end{lem}
\begin{proof}
Up to shrinking $B$, we can assume that $B$ is nonsingular and $\mathcal{S} \rightarrow B$ is an \'etale covering of degree $|\text{Sing}(X)|$. Then consider the base change
$$\xymatrix{\widetilde{\mathcal{X}} \ar[d]^{\widetilde{q}} \ar[r] & \mathcal{X} \ar[d]^{q} \\ \mathcal{S} \ar[r] & B}$$
By construction, there exists a section $\sigma:\mathcal{S} \rightarrow \widetilde{\mathcal{X}}$ of $\widetilde{q}$.\\
Let $\mathcal{N}_{\widetilde{X}/S}$ (respectively, $\mathcal{T}^1_{\widetilde{X}/S}$) be the relative normal bundle (respectively, the first relative cotangent sheaf) of $\widetilde{X}$ over $S$, and consider the map $\widetilde{q}_{*}\mathcal{N}_{\widetilde{X}/S} \xrightarrow{\rho} \sigma^{*}\mathcal{T}^1_{\widetilde{X}/S} \rightarrow \text{coker} \rho \rightarrow 0$.\\
Since $h^{0}(N'_X) < h^{0}(N_X)$, the image of the map $H^{0}(N_X) \rightarrow H^{0}(T^1_X)$ is at least one-dimensional, thus there is at least one $p \in \text{Sing}(X)$ such that the map $H^{0}(N_X) \rightarrow H^{0}({T^1_X}_{|p})$ is surjective. If $s=(q(\mathcal{X}_0),p) \in \mathcal{S}$, that map is exactly the fibre map
$${\widetilde{q}_{*}\mathcal{N}_{\widetilde{X}/S}}_{|s} \cong H^{0}(N_X) \xrightarrow{\rho_s} {\sigma^{*}\mathcal{T}^1_{\widetilde{X}/S}}_{|s} \cong H^{0}({T^1_X}_{|p})$$
where the isomorphisms are given by the fact that the cohomology of both sheaves is constant on fibres and both sheaves are flat over $\mathcal{S}$. Then, since $\text{coker} \rho$ is a coherent sheaf, there exists a nonempty open set $\mathcal{U} \subset \mathcal{S}$ such that ${\text{coker} \rho}_{|\mathcal{U}} \cong (0)$.
%La dimensione della fibra è superiormente semicontinua per fasci coerenti
Since $\mathcal{S}$ is irreducible, this implies that, for the general fibre $\widetilde{\mathcal{X}}_t$, the map $H^{0}(N_{\widetilde{\mathcal{X}}_t}) \rightarrow H^{0}({T^1_{\widetilde{\mathcal{X}}_t}}_{|w})$ is surjective for every $w \in \text{Sing}(\widetilde{\mathcal{X}}_t)$. Since $\text{Hilb}^r_{p_a(X),\deg(X)}$ is smooth at $[X]$, it is smooth at $[\widetilde{X}_t]$ too, hence $\widetilde{X}_t$ is smoothable.
%L'ipotesi h^0(N'_X) < h^0(N_X) è già stata usata precedentemente, per dimostrare la suriettività!
\end{proof}
\end{subsection}
\begin{subsection}{Elementary transformations of vector bundles}
Let $C \subset \mathbb{P}^r$ be a smooth irreducible curve and let $S := \left\{p_1,...,p_s\right\} \subset C$ be a smooth $0$-dimensional scheme of length $s$. Let $m$ be a positive integer and let $E,F$  be rank $m$ vector bundles on $C$ such that there exists an exact sequence
\begin{equation}
\label{pet}
0 \rightarrow E \xrightarrow{\alpha} F \rightarrow \mathcal{O}_S \rightarrow 0.
\end{equation}
The fibre map $\alpha_x:E_{x} \rightarrow F_x$ is injective for all $x \in C \smallsetminus S$, while it has a one-dimensional kernel $l_x \subset E_x$ for all $x \in S$. The vector bundle $F$ is thus uniquely determined by the scheme $S$ and by the lines $l_x \subset E_x$ for all $x \in S$. Let $v_i=\mathbb{P}(l_{p_i})$, $i=1,...,s$, and define $K := \left\{(p_i,v_i)\right\}_{i=1,...,s} \subset \mathbb{P}(E)$, where $\mathbb{P}(E)$ is the projectivized bundle of $E$.\\
We say that $F$ is the \emph{positive elementary transformation of $E$ defined by $K$}, and we denote it by $\text{Elm}^{+}_K(E)$. Of course, given the projection $\pi:\mathbb{P}(E) \rightarrow C$, one has $S=\pi(K)$.
\begin{prop}
\label{propertieselementary}
Notation as above, let $L \in \emph{Pic}(C)$. Fix an isomorphism $f:\mathbb{P}(E) \rightarrow \mathbb{P}(E \otimes L)$ and let $K^{\prime}$ be the image of $K$ under $f$. Then one has $\emph{Elm}^{+}_{K^{\prime}}(E \otimes L) \cong \emph{Elm}^{+}_K(E) \otimes L$.
\end{prop}
\begin{proof}
Twisting the exact sequence
$$0 \rightarrow E \xrightarrow{\alpha} \text{Elm}^{+}_K(E) \rightarrow \mathcal{O}_S \rightarrow 0$$
by $L$ one obtains
$$0 \rightarrow E \otimes L \xrightarrow{\beta} \text{Elm}^{+}_K(E) \otimes L \rightarrow \mathcal{O}_S \rightarrow 0.$$
The fibre map $\beta_x$ contracts exactly the same lines $l_x$ as $\alpha_x$ in the identification among $E_x$ and $(E \otimes L)_x$. The claim then follows by definition.
\end{proof}
\begin{prop}
\label{generalelementary}
Notation as above, let $K$ be sufficiently general. Then one has $h^1(\emph{Elm}^{+}_K(E))=\max \left\{0,h^1(E)-s\right\}$.
\end{prop}
\begin{proof}
It is sufficient to prove the result for $s=1$ i.e. $S=p$ and then iterate. Dualizing sequence (\ref{pet}) one obtains the short exact sequence
\begin{equation}
\label{petvee}
0 \rightarrow \left(\text{Elm}^{+}_K(E)\right)^{\vee} \rightarrow E^{\vee} \rightarrow \textbf{Ext}^1_{\mathcal{O}_C}(\mathcal{O}_p,\mathcal{O}_C) \rightarrow 0,
\end{equation}
where $\textbf{Ext}^1_{\mathcal{O}_C}(\mathcal{O}_p,\mathcal{O}_C) \cong \mathcal{O}_p$. Twisting sequence (\ref{petvee}) by $\omega_C$ one obtains the exact sequence
\begin{equation}
\label{petveeomega}
0 \rightarrow \left(\text{Elm}^{+}_K(E)\right)^{\vee} \otimes \omega_C \rightarrow E^{\vee} \otimes \omega_C \rightarrow \mathcal{O}_p \rightarrow 0.
\end{equation}
Since $K$ is general in $\mathbb{P}(E)$, one has
$$h^0(\left(\text{Elm}^{+}_K(E)\right)^{\vee} \otimes \omega_C)=\max \left\{h^0(E^{\vee} \otimes \omega_C)-1,0\right\}.$$
The claim then immediately follows by Serre duality.
\end{proof}
Let now $X \subset \mathbb{P}^r$ be a connected reduced nodal curve having $C$ as an irreducible component, and such that $C$ intersects $\overline{X \smallsetminus C}$ in a smooth $0$-dimensional subscheme $\Delta := \left\{p_1,...,p_s\right\}$ of length $s$.
\begin{prop}
\label{isonormalelementary}
Notation as above, let $l_i$, $i=1,...,s$, be the tangent line to $\overline{X \smallsetminus C}$ at the point $p_i$, let $v_i=\mathbb{P}(l_i)$ and let $K := \left\{(p_i,v_i)\right\}_{i=1,...s} \subset \mathbb{P}(N_C)$. Then one has that ${N_{X}}_{|C} \cong \emph{Elm}^{+}_K(N_C)$.
\end{prop}
\begin{proof}
For all $p_i \in \Delta$, the kernel of the fibre map $\alpha_{p_i}:\left({N_{C}}\right)_{p_i} \rightarrow \left({{N_X}_{|C}}\right)_{p_i}$ (see (\ref{elementarytransformationnci})) is exactly the tangent line to $\overline{X \smallsetminus C}$ at the point $p_i$.
\end{proof}
\begin{cor}
\label{generalpivi}
Notation as in Proposition \ref{isonormalelementary}, let $L \in \emph{Pic}(C)$ and assume that $K := \left\{(p_i,v_i)\right\}_{i=1,...,s} \subset \mathbb{P}(N_C(L))$ is general. Then $h^1({N_{X}}_{|C}(L))=\max\left\{0,h^{1}(N_C(L))-s\right\}$.
\end{cor}
\begin{proof}
Combine Proposition \ref{propertieselementary}, Proposition \ref{generalelementary} and Proposition \ref{isonormalelementary}.
\end{proof}
\end{subsection}
\begin{subsection}{Secant spaces to a projective curve}
Let $C$ be a smooth curve of genus $g$, let $L \in \text{Pic}(C)$ and let $l=\mathbb{P}(V)$, $V \subset H^{0}(C,L)$ be a linear series on $C$, $\dim l=r \geq 2$. Let $0 \leq f < e$ be integers, let $C_{(e)}$ be the $e$-symmetric product of $C$ and consider the locus $V^{e-f}_e(l)$ defined set-theoretically as
$$V^{e-f}_e(l) = \left\{D \in C_{(e)} \, \big| \, \dim l(-D) \geq r-e+f \right\}.$$
This is the locus of effective divisors of degree $e$ on $C$ which impose at most $e-f$ independent conditions to $l$. If $l$ is very ample and $C$ is identified with its image in the embedding $C \hookrightarrow_{\varphi_{l}} \mathbb{P}^r$, then $V^{e-f}_e(l)$ parameterizes $(e-f-1)$-planes in $\mathbb{P}^r$ which are $e$-secant to $C$. We will write $V^{e-f}_{e}(L)$ for $V^{e-f}_e(\mathbb{P}(H^{0}(C,L)))$.\\
There is an equivalent definition, which generalizes to the case in which $C$ is singular. Let $Y$ be an algebraic scheme, let $\mathcal{L} \in \text{Pic}(Y)$, $\mathfrak{l}=\mathbb{P}(\mathcal{V})$, $\mathcal{V} \subset H^{0}(Y,\mathcal{L})$ be a linear series, and consider the commutative diagram
$$\xymatrix{Y_{(e-1)} \times Y \ar[d]_{f} \ar@{^{(}->}[r]^{i} & Y_{(e)} \times Y \ar[ld]_{p} \ar[r]^{q} & Y \\ Y_{(e)}  & & }$$
where the maps are the obvious ones. Define $E_{\mathcal{L}} := f_{*}i^{*}q^{*}\mathcal{L}$. By \cite[(2.5)]{SCHW}, there is a naturally defined morphism of sheaves
$$\sigma_\mathcal{V}:\mathcal{V} \otimes \mathcal{O}_{Y_{(e)}} \rightarrow E_{\mathcal{L}}.$$
We will write $\sigma_{\mathcal{L}}$ for $\sigma_{H^{0}(Y,\mathcal{L})}$.\\
Now, let $C$ be a reduced but possibly singular curve, let $L \in \text{Pic}(C)$, $l=\mathbb{P}(V)$, $V \subset H^{0}(C,L)$, let $C_{e}$ be the $e$-cartesian product of $C$ and consider the obvious projections
$$C_{(e)} \xleftarrow{h} C_{e} \xrightarrow{\pi_i} C.$$
Set $L[e]  := \bigoplus_{i=1}^{e} \pi_i^{*}L$. Clearly, $L[e]$ is a locally free sheaf on $C_e$ of rank $e$. Let $\mathcal{E}_L$ be the sheaf on the $e$-th symmetric product of the smooth locus of $C$, $C^{\textrm{sm}}_{(e)}$, whose sections on an open set $U$ are the $G$-invariant sections of $L[e]$ on $h^{-1}(U)$ where $G$ is the Galois covering-map group of $h$. It can be shown (see \cite[Proposition 1 and 2]{MA}) that over $C^{\textrm{sm}}_{(e)}$ one has $E_L \cong \mathcal{E}_L$, that $\mathcal{E}_
 L$ is a locally free sheaf of rank $e$, and the map $\sigma_V$ restricts over $C^{\textrm{sm}}_{(e)}$ to a morphism of vector bundles. Define
\begin{equation}
\label{secondadefvefe}
V^{e-f}_e(l) = \left\{D \in C^{\textrm{sm}}_{(e)} \, \big| \, \text{rk } {\sigma_{V}}_{|D} \leq e-f \right\}.
\end{equation}
Scheme-theoretically, $V^{e-f}_e(l)$ is given locally at $D$ by (the vanishing of) the $e-f+1$ minors of some matrix representation of $\sigma_{V}$ at $D$ (see \cite[Section 4]{HJ} for details). An immediate consequence of the definition is the following
\begin{prop}
Notation as above, either $V^{e-f}_e(l)$ is empty or for every irreducible component $V \subset V^{e-f}_e(l)$ one has
\begin{equation}
\label{expvefe}
\dim V \geq e-f(r+1-e+f).
\end{equation}
\end{prop}
If equality holds in (\ref{expvefe}), we say that $V$ has \emph{expected dimension}.\\
The case we are interested in is the case in which $L$ is very ample, $l$ is complete and $f=1$ (we will write $V^{e-1}_e(L)$ instead of $V^{e-1}_e(l))$. In this situation, the linear span of a divisor $D \in V^{e-1}_e(L)$ is a $(e-2)$-plane, and the right term of (\ref{expvefe}) becomes $2e-r-2$, thus for $e < \frac{r}{2}+1$ we expect the emptyness of $V^{e-1}_e(L)$. However, even if we impose the condition $e \geq \frac{r}{2}+1$ and rule out trivial cases such as the one of the rational normal curve, the non-emptyness of $V^{e-1}_e(L)$ is not granted in general. In the following Proposition we show non-emptyness for curves obtained by smoothing a suitable reducible curve. Given a nondegenerate curve $C \subset \mathbb{P}^r$ and a positive integer $k$, define the \emph{variety of $k$-secants of $C$} $S^k(C)$ as the closure in $\mathbb{P}^{r}$ of the quasiprojective variety
$$\mathring{S}^k(C) := \left\{p \in <x_0,...,x_k> | x_0,...,x_k \in C, \,\dim <x_0,...,x_k>=k\right\}.$$
The variety $S^k(C)$ is irreducible of dimension $\min \left\{2k+1,r\right\}$.\\
In what follows we will write ``l.g.p." for ``linearly general position" and, with a little abuse of notation, we will indicate with the same notation a certain divisor on a curve and its support.
\begin{prop}
\label{r2secant}
%Let $C \subset \mathbb{P}^2$ be a linearly normal curve of degree $d \geq 2$. Let $\Gamma_2 := C$ and, for $r \geq 3$, let $Z_r := \Gamma_{r-1} \cup$
Let $r \geq 2$ and let $C \subset H \cong \mathbb{P}^r \subset \mathbb{P}^{r+1}$ be a smooth (nondegenerate) linearly normal curve.
%such that $[C]$ is a general point of an irreducible component of $\emph{Hilb}^r_{g(C),\deg(C)}$ of dimension greater than or equal to $r^2+r-2$.
%Come spiegato nei commenti alla proof, questa assunzione è ridondante, perché $r^2+r-2$ è inferiore al numero delle proiettività.
%le assunzioni sul grado sono contenute nell'assunzione che $V^{e-1}_{e}(\mathcal{O}_B(1))$ sia non vuoto, per esempio, nel caso della curva razionale normale, è vuoto.
Assume that, for some $\lfloor \frac{r}{2} \rfloor +3 \leq e \leq r+2$, the scheme $V^{e-1}_{e}(\mathcal{O}_C(1))$ is nonempty and has an irreducible component $V_e$ whose general point parameterizes a divisor $S^r_e=S_e \in \emph{Div}(C)$ satisfying the following properties:
%tutte le proprietà di star e servono anche per la costruzione di X.
%nel caso e minimo V_e ha dimensione 0, generale significa il punto stesso.
$$(\star)_e \quad \left\{\begin{array}{ll}
               (i) & S_e \in V^{e-1}_e(\mathcal{O}_{C}(1)) \smallsetminus V^{e-2}_e(\mathcal{O}_{C}(1)); \\
               (ii) & S_e \hbox{ is smooth}; \\
               (iii) & \hbox{the points of } S_e \hbox{ are in l.g.p. in } <S_e>.
               %(iv) & V_e \hbox{ has expected dimension } (2e-r-2).
             \end{array}\right.$$
Let $H^{\prime} \cong \mathbb{P}^{e-1} \subset \mathbb{P}^{r+1}$ be a $(e-1)$-plane cutting on $H$ the $(e-2)$-plane $<S_e> := H_{e-2}$.
%, and thus $h$-secant to $B$ in a smooth $0$-dimensional subscheme $S := \left\{p_1,...,p_h\right\}$, falso per h=r+2.
Let $E \subset H^{\prime}$ be a smooth nondegenerate curve intersecting $C$ transversally in $S_e$, and such that $[E]$ is a general point of an irreducible component of $\emph{Hilb}^{e-1}_{g(E),\deg(E)}$.
% of dimension greater than or equal to $e^2-4$. ridondante, perché e^2-4=(e+2)(e-2) è minore del numero di proiettività.
%L'assunzione nondegenerate esclude in caso E \subset H_{e-2}.
%La curva deve avere grado h in H^{\prime}, quindi assumere che sia ellittica non sarebbe una perdita di generalità mostruosa.
Assume that $X := C \cup E \subset \mathbb{P}^{r+1}$ is linearly normal and smoothable, and let $\Gamma \subset \mathbb{P}^{r+1}$ be a smooth curve such that $[\Gamma]$ is a general point of an irreducible component $W \subset \emph{Hilb}^{r+1}_{p_a(X),\deg(X)}$ containing the point $[X]$. Then the curve $\Gamma$ satisfies the following property:
$$(\star \star)_{r+1} \quad \begin{array}{ll} \hbox{ for every } \lfloor \frac{r+1}{2} \rfloor + 2 \leq h \leq r+3,\hbox{ the scheme } V_{h}^{h-1}(\mathcal{O}_{\Gamma}(1)) \hbox{ is nonempty, and it has} \\
\hbox{ an irreducible component } V_h \hbox{ of the expected dimension } (2h-(r+1)-2) \hbox{ satisfying }\\
\hbox{ the following properties:}\end{array}$$
%\left\{\begin{array}{ll}
%               (i) & \hbox{ the general point } S^{r+1}_h=S_h \in V_h \hbox{ satisfies } (\star)_h \hbox{ relative to } r+1; \\
%               (ii) & \hbox{ for every } \lfloor \frac{r+1}{2} \rfloor +3 \leq e < h \hbox{ there exists a divisor } S_{e,h} \in V_h \hbox{ such that }\\
%               & S_{e,h}=S^{r+1}_e+M_e \hbox{ with } M_e := p^{\prime}_{e+1}+ \dots + p^{\prime}_{h} \hbox{ effective}.
%             \end{array}\right.
\begin{itemize}
\item[$(i)$] the general point $S^{r+1}_h=S_h \in V_h$ satisfies properties $(\star)_h$;
\item[$(ii)$] for every $\lfloor \frac{r+1}{2} \rfloor +2 \leq f < h$, one has $V_f + C_{(h-f)} \subset V_h$.
 %$\mathring{V}_f+\mathring{C}_{(h-f)} \subset V_h$, where $\mathring{V}_f$ and $\mathring{C}_{(h-f)}$ are the nonempty open subsets of $V_f$ and $C_{(h-f)}$, respectively, parameterizing general divisors.
 %però, effettivamente, essendo V_h una varietà, anche i limiti sono contenuti in V_h.
%there exists a divisor $S_{f,h} \in V_h$ such that $S_{f,h}=S^{r+1}_f+T_f$ with $T_f := p^{\prime}_{f+1}+ \dots + p^{\prime}_{h}$ and $p^{\prime}_{i}$, $i=f+1,...,h$, general.
%La definizione precedente, in cui (**) era (**)_h i.e. dipendeva da h, aveva lo svantaggio di non esprimere chiaramente che i V_h inscatolati erano sempre gli stessi, indipendentemente da h.
%prendiamo gli aperti per non dover considerare i divisori non ridotti.
\end{itemize}
%questa definizione ha il vantaggio di non distinguere il caso h-secant dal caso "almeno h"-secant.
%il fatto che la componente abbia expected codimension serve solo per la dimostrazione.
%Define $H_{r} := H$ and $S_{r+2} := S_{r+1} \cup p_{r+2}$, where $p_{r+2}$ is a general point of $C$.
\end{prop}
\begin{proof}
We indicate by the symbol $/ \sim$ the quotient by the permutation group. Let $M_h$ be a general $(h-2)$-secant (to $C$) $(h-3)$-plane. If $e=r+2$, one has $<M_h,H_{e-2}>=H_{e-2}=\mathbb{P}^r$. Assume that $e \leq r+1$. Note that there exists a projectivity of $\mathbb{P}^r$ sending a set $U_1$ of $r+2$ points whose every proper subset consists of linearly independent points to another set $U_2$ with the same property. Now, taking $U_1:=\left\{S_e,y_1,...,y_{r-(e-2)}\right\}$, where the $y_i$'s are general points on $C$, and $U_2:=\left\{S_e,z_1,...,z_{r-(e-2)}\right\}$, where the $z_i$'s are general points of $\mathbb{P}^r$, we obtain that there is a curve projectively equivalent to $C$, call it again $C$, containing $U_2$. Now, let $N:=<z_1,...,z_{r-(e-2)}>$. One has $<N,H_{e-2}> \cong \mathbb{P}^r$. Since $r-(e-2) \leq h-3$ in our range, one can assume $M_h$ to contain $N$, thus $<M_h,H_{e-2}> =\mathbb{P}^r$.
%It is indeed sufficient to impose that $C$ passes through $\underbrace{e}_{=|S_e|}$+$r-(e-2)$ points, which accounts for at most $(r+2)(r-1)=r^2+r-2$ conditions.
%Se H_{e-2} ha dimensione r non occorre imporre alcun passaggio, se H_{e-2} ha dimensione r-1 occorre imporre il passaggio per un punto esterno ecc.
%Ma siccome il numero ottenuto è inferiore alle proiettività e le proiettività conservano la dimensione dei sottospazi generati, non occorre imporre alcuna condizione.
The projective Grassmann formula then gives
\begin{equation}
\label{grassmann}
\dim (M_h \cap H_{e-2})=\dim M_h + \dim H_{e-2}-\dim (M_h+H_{e-2})=h-3+e-2-r=e+h-5-r.
\end{equation}
Consider the (irreducible) incidence variety
$$A := \left\{((x_0,...,x_{h-3}),p) \in C_{(h-2)} \times H_{e-2} \, \bigg| \, p \in <x_0,...,x_{h-3}>\right\}$$
with the two canonical projections $\pi_1:A \rightarrow C_{(h-2)}$ and $\pi_2:A \rightarrow H_{e-2}$.
%C_{(h-2)} è irriducibile e la fibra di pi_1 è un sottospazio lineare di H_{e-2}, quindi A è irriducibile.
Let $F$ be a general fibre of $\pi_2$.
By (\ref{grassmann}), the general fibre of $\pi_1$ is $(e+h-5-r)$-dimensional, hence one has $\dim A=\dim C_{(h-2)}+(e+h-5-r)=e+2h-7-r$. One has $\dim S^{h-3}(C)=r-1$ if $h=\lfloor \frac{r+1}{2} \rfloor +2$ and $r$ is even, $S^{h-3}(C)=\mathbb{P}^r$ otherwise. Note that $S^{h-3}(C) \cap H_{e-2}$ and $S^1(E) \cap H_{e-2}$ are both equidimensional of respective dimension $\dim \text{im}(\pi_2)$ and $\dim S^1(E)-1=2$. We have two possible cases:
%attenzione, i due casi non corrispondono necessariamente ai due casi sopra. Non mi voglio impegnare nell'affermazione che che S^h-3(C) non contiene H_e-2 se dim S^h-3(C)=r-1.
\begin{itemize}
  \item[a)] $\text{im}(\pi_2)=H_{e-2}$, which gives $\dim F=\dim A - \dim \text{im}(\pi_2)=2h-5-r$. Moreover, one has $S^1(E) \cap S^{h-3}(C)=S^1(E) \cap H_{e-2}$. Let $T$ be an irreducible component of that scheme.
  \item[b)] $\dim \text{im}(\pi_2)=e-3$, which gives $\dim F=\dim A - \dim \text{im}(\pi_2)=2h-4-r$. The gene-rality assumption on $[E]$ implies that $S^1(E)$ contains the general point of $H_{e-2}$, hence there is at least one irreducible component $T \subset S^1(E) \cap S^{h-3}(C)$ such that $\dim T=\dim \left(S^1(E) \cap H_{e-2}\right)-1=1$.
\end{itemize}

%L'argomento che segue esigeva impegnarsi nell'affermazione che S^h-3(C) non contiene H_e-2 se dim S^h-3(C)=r-1, cosa che non voglio fare.
%Then
%$$\dim \text{im}(\pi_2)=\left\{\begin{array}{ll}
%e-3, & h=\overline{h}, r \hbox { even } \\
%e-2, & \hbox{ otherwise }
%\end{array}\right.$$
%and, if $F$ is the general fibre of $\pi_2$
%$$\dim F=\dim A - \dim \text{im}(\pi_2)=\left\{\begin{array}{ll}
%2h-4-r, & h=\overline{h}, r \hbox{ even } \\
%2h-5-r, & \hbox{ otherwise }
%\end{array}\right.$$
%Let $T:=S^1(E) \cap S^{h-3}(C)$. Since $\dim \left(S^1(E) \cap H_{e-2}\right)=3-1=2$ one has
%$$\dim T=\left\{\begin{array}{ll}
%1, & h=\overline{h}, r \hbox { even } \\
%2, & \hbox{ otherwise }
%\end{array}\right.$$
%fissati e punti devo imporre il passaggio per altri 2 punti in modo che la retta secante sia generica.
%Let $T^{\prime} \subset T$ be an irreducible component. Non serve, perché nel seguito prendo già le c.irr.
In both cases, let $B_h$ be the irreducible component of the locus
$$\left\{(x_0,...,x_{h-3},q_1,q_2) \in \frac{\left(C \smallsetminus S_e \right)_{(h-2)} \times \left(E \smallsetminus S_e \right)_{(2)}}{\sim} \, \bigg| \, <q_1,q_2> \cap T \in <x_0,...,x_{h-3}> \right\}$$
such that, for a general (h-1)-tuple $(x_0,...,x_{h-3},q_1,q_2) \in B_h$, $(x_0,...,x_{h-3})$ is a general fibre of $\pi_2$.
%Ci potrebbe essere un numero finito > 1 di rette 2-secanti E e passanti per p \in T, per cui il luogo sopra potrebbe non essere irriducibile.
%D'altra parte, la riducibilità viene SOLO da questa possibilità, data l'irriducibilità di T e di
%The generality assumption on $[E]$ implies that $T$ contains the general point of $\text{im}(\pi_2)$, so that the general $(h-1)$-tuple $(x_0,...,x_{h-3},q_1,q_2)$ has distinct entries and $(x_0,...,x_{h-3})$ is a general fibre of $\pi_2$.\\
One has $B_h \subset V^{h-1}_h(\mathcal{O}_X(1))$ and $B_h \nsubseteq V^{h-2}_h(\mathcal{O}_X(1))$. Suppose that $B_h \subseteq B^{\prime}_h$, where $B^{\prime}_h$ is an irreducible component of $V^{h-1}_h(\mathcal{O}_X(1))$. Since the multidegree over the irreducible components of $X$ of a divisor $D \in B^{\prime}_h$ remains constant over $B^{\prime}_h$, and, for a general divisor $D_h \in {B_h}$, $\dim <{D_h}_{|C}>$ and $\dim <{D_h}_{|E}>$ are the maximal possible dimensions, it must be $B_h=B^{\prime}_h$.  We claim that there is only a finite number of $2$-secant lines to $E$ passing through a general point of $T$. Suppose by contradiction that this in not the case. Then the scheme of 2-secants to $E$ passing through a point of $T$ has an irreducible component of dimension $\dim T+2$, hence, if $\dim T=1$, it must coincide with $S^1(E)$ since $S^1(E)$ is irreducible of dimension 3, contradiction as $\dim \left(S^1(E) \cap H_{e-2}\right)=2$. If $\dim T=2$, contradiction immediately follows as well. As a consequence, $B_h$ has dimension $\dim T+\dim F=2h-3-r=2h-(r+1)-2$, which is the expected one.\\
Let $\mathcal{X}_{(h)}$ be the $h$-symmetric product of the universal family $\mathbb{P}^{r+1} \times W \supset \mathcal{X} \xrightarrow{\Psi} W$, let $\mathcal{L} := \mathcal{O}_{\mathcal{X}}(1) \in \text{Pic}(\mathcal{X})$ be the relative hyperplane bundle and let
$$\sigma_{\mathcal{L}}:H^{0}(\mathcal{X},\mathcal{L}) \otimes \mathcal{O}_{\mathcal{X}_{(h)}} \rightarrow E_{\mathcal{L}}$$
be the map defined above. Let $\widetilde{\mathcal{X}}$ be the restriction of $\mathcal{X}$ to the smooth locus of the fibres of $\Psi$, let $\widetilde{\mathcal{X}}^{\text{rel}}_{(h)} \xrightarrow{\nu} W$ be the relative $h$-symmetric product of $\widetilde{\mathcal{X}} \xrightarrow{\widetilde{\Psi}} W$
%che non è il luogo liscio di \mathcal{X}_{d}, è più piccolo!
and define $\mathcal{V}^{h-1}_h(\mathcal{L}) := \left\{ D \in \widetilde{\mathcal{X}}^{\text{rel}}_{(h)} \, \big| \, \text{rk} \, {\sigma_{\mathcal{L}}}_{|D} \leq h-1 \right\}$. From the definition, it follows that either $\mathcal{V}^{h-1}_h(\mathcal{L})$ is empty or each of its irreducible components has dimension greater than or equal to $2h-(r+1)-2+\dim W$.\\
Let $\Gamma$ be $X$ or a general fibre of $\Psi$, let $L := \mathcal{O}_{\Gamma}(1)$ and let $D \in \Gamma_{(h)}$. Since ${\sigma_{\mathcal{L}}}_{|D}={\sigma_{L}}_{|D}$, one has $\mathcal{V}^{h-1}_h(\mathcal{L})_{|\nu^{-1}([\Gamma])} \cong V^{h-1}_h(L)$.\\
%Non posso prendere una fibra speciale G perché potrebbe non essere linearmente normale, e quindi potrebbe essere h^{0}(G) > h^{0}(\mathcal{L}).
Since $B_h \subset V^{h-1}_{h}(\mathcal{O}_X(1))$ has expected dimension, it must ``extend" to an irreducible component $\mathcal{B} \subset \mathcal{V}^{h-1}_h(\mathcal{L})$ intersecting $V^{h-1}_h(\mathcal{O}_{\Gamma}(1))$, otherwise there would be an irreducible component of $\mathcal{V}^{h-1}_h(\mathcal{L})$ of dimension less than $2h-(r+1)-2+\dim W$, which is not possible. This yields the existence of an irreducible component $V_h \subset V^{h-1}_h(\mathcal{O}_{\Gamma}(1))$ of the expected dimension. Since, by construction, the general point $D^{r+1}_h = D_h \in B_h$ is a smooth divisor of $X$, the general point $S_h \in V_h$ must be a smooth divisor too. Secondly, the points of $D_h$ are in l.g.p. in $<D_h>$, hence the points of $S_h$ satisfy the same property in $V_h$. Moreover, by construction, for every $\lfloor \frac{r+1}{2} \rfloor +2 \leq f <h$, there exists a divisor $D_{f,h} \in B_h$ such that $D_{f,h}=D^{r+1}_f+N_f$ and $N_f \in \text{Pic}(C)$ is a general divisor of degree $h-f$. Indeed, if $M_f$ is the $(f-2)$-secant (to $C$) $(f-3)$-space defining $D^{r+1}_f$, the space $<M_f,N_f>$ is a $(h-2)$-secant (to $C$) $(h-3)$-space containing $M_f \cap T$. Assertion $(ii)$ of $(\star \star)_{r+1}$ then follows from deforming $D_{f,h}$.
%I punti di N_f sono situati sulla sola componente C, ma comunque si muovono in una varietà (h-f)-dimensionale, quindi anche i punti di T_f si muovono in una varietà (h-f)-dimensionale i.e. sono generali su $\Gamma$.
%completo uno spazio f-2-secante C a uno spazio h-2-secante C!
\end{proof}
\begin{oss}
\label{r2secantr3}
We can add to Proposition \ref{r2secant} the additional case $r=2, e=3$, with property $(\star \star)_3$ holding in the restricted range $4 \leq h \leq 5$. The proof is identical to the one of the Proposition, except for the fact that when $h=4$ one has that $S^1(E) \cap H_1=H_1$ is a line, $T=H_1$ and the variety of 2-secant (indeed, 3-secant) lines to $E \subset \mathbb{P}^2$ passing through a point of $T$ is obviously 1-dimensional.
\end{oss}
%\begin{oss}
%\label{rquadroquarti}
%Let $G_{h}(\mathbb{P}^r)$ be the grassmannian of linear $h$-subspaces of $\mathbb{P}^r$. Then $\dim G_h(\mathbb{P}^r)=(h+1)(r-h)$. Imposing that a $h$-plane intersects $C$ in $h+2$ points accounts for at most $(h+2)(r-h-1)$ conditions, thus either $V^{h+1}_{h+2}(\mathcal{O}_C(1))$ is empty or every irreducible component of $V^{h+1}_{h+2}(\mathcal{O}_C(1))$ has dimension at least $2h+2-r$.\\
%Consequently, if $h \geq \frac{r}{2}-1$, the non-emptiness of $V^{h+1}_{h+2}(\mathcal{O}_C(1))$ (i.e. the existence of a $h$-plane which is %$h+2$-secant to $C$) is in general expected.
%\end{oss}
%if $(h+1)(r-h) \geq (h+2)(r-h-1)$, there exists a $h$-plane which is $h+2$-secant to $B$. Solving the inequality one obtains $h \geq \lfloor \frac{r}{2} \rfloor$. For $h=\lfloor \frac{r}{2} \rfloor$, let $H_h$ such a space, for $\lfloor \frac{r}{2} \rfloor +1 \leq h \leq r$ let $H_h := <H_{h-1},p_h>$, where $p_h$ is a general point of $B$ for all $h$.
\end{subsection}
\end{section}
\begin{section}{Technical lemmas for the proof of the main result}
\label{technicallemmas}
\begin{lem}
\label{h1nd-delta}
Let $r \geq 2$, let $E \subset \mathbb{P}^r$ be an elliptic normal curve, let $q$ be a point of $E$ and let $\Sigma \subset \mathbb{P}^r$ be hyperplane section of $E$. Then one has $H^{1}(N_E(-\Sigma - q))=(0)$. In particular, $h^1(N_E(-\Sigma))=0$.
\end{lem}
%con la nuova dimostrazione si può far cadere l'assunzione "sufficiently general" e "general" per il punto q.
\begin{proof}
By \cite[Theorem 4.1]{EL} one has that $N_E(-\Sigma - q)$ is a semistable vector bundle. Let $\mu(N_E(-\Sigma -q))$ be the slope of $N_E(-\Sigma -q)$. One has $\mu(N_E(-\Sigma - q))=\frac{r+3}{r-1} >0$, hence every quotient line bundle of $N_E(-\Sigma -q)$ has degree $>0$, and there is no nonzero morphism $N_E(-\Sigma -q) \rightarrow \omega_E$. By Serre duality, one obtains $h^1(N_E(-\Sigma -q))=0$.
\end{proof}
\begin{lem}
\label{ehlsmoothable}
Let $r \geq 2$, let $E \subset H \cong \mathbb{P}^{r} \subset \mathbb{P}^{r+1}$ be an elliptic normal curve, let $l$ be a line intersecting $E$ at a point $q$ and let $X := E \cup l$. Then $X$ is smoothable to an elliptic normal curve in $\mathbb{P}^{r+1}$. In particular, $[X]$ is a point of the irreducible component of $\emph{Hilb}^{r+1}_{1,r+2}$ which dominates $\mathcal{M}_1$.
\end{lem}
%L'enunciato precedente specificava $l \subsetneq H$, ma non serve. O meglio, garantiva che E interseca l SOLO in q. Scritto così, però, era un po' fuorviante. Nel nuovo enunciato ho lasciato implicito che E e l si intersechino SOLO in q.
\begin{proof}
Let $\Sigma$ be a hyperplane section of $E$. The exact sequence
$$0 \rightarrow N_{E/H} \rightarrow N_{E/\mathbb{P}^{r+1}} \rightarrow \mathcal{O}_E(\Sigma) \rightarrow 0$$
immediately gives $h^1(N_{E/\mathbb{P}^{r+1}})=0$.
Since $N_l \cong \mathcal{O}_{\mathbb{P}^1}(1)^{ \oplus r}$, sequence (\ref{elementarytransformationnci}) for $l$ twisted by $\mathcal{O}_l(-q)$ gives $h^1({N_{X}}_{|l}(-q))=0$.
The short exact sequence (see (\ref{icinxnprimex}))
$$0 \rightarrow {N_X}_{|l}(-q) \rightarrow N^{\prime}_X \rightarrow N_{E/\mathbb{P}^{r+1}} \rightarrow 0$$
then gives $h^1(N^{\prime}_X)=0$, hence $X$ is smoothable to an elliptic normal curve in $\mathbb{P}^{r+1}$.
\end{proof}
\begin{lem}
\label{i2}
Let $B \subset A \subset Y$ be algebraic schemes such that both $A$ and $B$ are regularly embedded in $Y$ (see \emph{\cite[Section D.1]{SE}} for a precise definition) and let $\mathcal{I}_A \subset \mathcal{I}_B \subset \mathcal{O}_Y$ be the ideal sheaves of $A$ and $B$ in $Y$. Then there exists an exact sequence
\begin{equation}
0 \rightarrow \mathcal{I}^2_A \rightarrow \mathcal{I}^2_B \rightarrow \mathcal{A} \rightarrow 0
\end{equation}
where the sheaf $\mathcal{A}$ fits inside the exact sequence
\begin{equation}
0 \rightarrow \mathcal{I}_{B/A} \otimes N^{\vee}_{A/Y} \rightarrow \mathcal{A} \rightarrow \mathcal{I}^2_{B/A} \rightarrow 0.
\end{equation}
\begin{proof}
Consider the commutative exact diagram
\begin{equation}
\xymatrix{& 0 \ar[d] & 0 \ar[d] & \mathcal{B} \ar[d] &\\ 0 \ar[r] & \mathcal{I}^2_A \ar[r] \ar[d] & \mathcal{I}^2_B \ar[r] \ar[d] & \mathcal{A} \ar[r] \ar[d] & 0\\ 0 \ar[r]  & \mathcal{I}_A \ar[r] \ar[d] & \mathcal{I}_B \ar[r] \ar[d] & \mathcal{I}_{B/A} \ar[r] \ar[d] & 0 \\& N^{\vee}_{A/Y} \ar[r] \ar[d] & N^{\vee}_{B/Y} \ar[r] \ar[d] & N^{\vee}_{B/A} \ar[d] \ar[r] & 0 \\ & 0 & 0 & 0 &}
\end{equation}
where the quotient $\mathcal{A}/\mathcal{B}$ is isomorphic to $\mathcal{I}^2_{B/A}$.\\
The snake lemma gives the existence of a 4-term exact sequence
\begin{equation}
\label{conseq}
0 \rightarrow \mathcal{B} \rightarrow N^{\vee}_{A/Y} \rightarrow N^{\vee}_{B/Y} \rightarrow N^{\vee}_{B/A} \rightarrow 0.
\end{equation}
Since $B$ and $A$ are both regularly embedded in $Y$, by the proof of \cite[Lemma D.1.3 $(ii)$]{SE} sequence (\ref{conseq}) splits into the two short exact sequences $0 \rightarrow \mathcal{B} \rightarrow N^{\vee}_{A/Y} \rightarrow {N^{\vee}_{A/Y}}_{|B} \rightarrow 0$ and
$0 \rightarrow {N^{\vee}_{A/Y}}_{|B} \rightarrow N^{\vee}_{B/Y} \rightarrow N^{\vee}_{B/A} \rightarrow 0$.
%la sequenza (D.2) esiste indipendentemente dal fatto che gli embedding siano regolari. Nella mia situazione non ho B \subset A embedding regolare, ma questo non mi interessa. Il fatto che la mappa a di (D.4) sia iniettiva necessita solo del fatto che A \subset Y e B \subset Y siano regolari.
%Se non siamo nelle ipotesi sopraddette, il kernel della seconda successione potrebbe non essere {N^{\vee}_{A/Y}}_{|B}.
This gives in turn $\mathcal{B} \cong \mathcal{I}_{B/A} \otimes N^{\vee}_{A/Y}$.
\end{proof}
\end{lem}
\begin{lem}
\label{bdegenerate}
Let $B \subset H \cong \mathbb{P}^r \subset \mathbb{P}^{r+1}$ be a linearly normal curve. If $H^{1}(\mathcal{I}^2_{B/H}(2))=(0)$, then $H^{1}(\mathcal{I}^2_{B/\mathbb{P}^{r+1}}(2))=(0)$.
\end{lem}
\begin{proof}
By Lemma \ref{i2}, there exists an exact sequence
$$0 \rightarrow \mathcal{I}^2_{H/\mathbb{P}^{r+1}}(2) \rightarrow \mathcal{I}^2_{B/\mathbb{P}^{r+1}}(2) \rightarrow \mathcal{A} \rightarrow 0$$
where $0 \rightarrow \mathcal{B} \rightarrow \mathcal{A} \rightarrow \mathcal{I}^2_{B/H}(2) \rightarrow 0$ and $\mathcal{B} \cong \mathcal{I}_{B/H} \otimes N^{\vee}_{H/\mathbb{P}^{r+1}}(2) \cong \mathcal{I}_{B/H}(1)$.\\
Since $B$ is linearly normal, one has $H^{1}(\mathcal{I}_{B/H}(1))=(0)$. Noting that $\mathcal{I}^2_{H/\mathbb{P}^{r+1}}(2) \cong \mathcal{O}_{\mathbb{P}^{r+1}}$, and computing cohomology the claim follows.
\end{proof}
\begin{lem}
\label{nye}
Fix integers $r \geq 2$, $\lfloor \frac{r}{2} \rfloor +2 \leq h \leq r+2$. Let $B \subset H \cong \mathbb{P}^r \subset \mathbb{P}^{r+1}$ be a smooth (nondegenerate) curve carrying a smooth divisor $S=S_h := \left\{p_1,...,p_h\right\}$ satisfying $(\star)_h$ of Proposition \ref{r2secant}. Let $H^{\prime} \cong \mathbb{P}^{h-1} \subset \mathbb{P}^{r+1}$ be a $(h-1)$-plane cutting on $H$ the $(h-2)$-plane $<S_h> := H_{h-2}$.
%, and thus $h$-secant to $B$ in a smooth $0$-dimensional subscheme $S := \left\{p_1,...,p_h\right\}$, falso per h=r+2.
Let $E \subset H^{\prime}$ be an elliptic normal curve intersecting $B$ transversally in $S_h$, let $\Sigma$ be a hyperplane section of $E$ and let $Y := B \cup E \subset \mathbb{P}^{r+1}$. Let $l_i$, $i=1,...,h$, be the tangent line to $B$ at $p_i$, let $v_i=\mathbb{P}(l_i)$ and assume that, for $h \leq r+1$ and $k=r+2-h$, the $0$-dimensional scheme $K := \left\{(p_i,v_i)\right\}_{i=1,...,k} \subset \mathbb{P}(N_{E/\mathbb{P}^{r+1}}(-\Sigma))$ is general. Then one has $H^{1}({N_Y}_{|E}(-\Sigma))=(0)$.
\end{lem}
%da come è scritto il Lemma ,sembrerebbe che l'ipotesi h \geq lfloor \frac{r}{2}+2 \rfloor +3 sia inutile, in realtà se h + troppo piccolo, k diventa troppo grande. Nella versione precedente c'era \frac{r}{2}+3 al posto di \frac{r}{2}+2, l'importante è che sia h > r+2-h, altrimenti non riesco ad annullare h^1.
\begin{proof}
Consider the commutative exact diagram
\begin{equation}
\label{ehprimeepr+1}
\xymatrix{ & 0 \ar[d] & 0 \ar[d] & &\\ 0 \ar[r] & N_{E/H^{\prime}} \ar[r]^{\cong} \ar[d] & N_{E/H^{\prime}} \ar[r] \ar[d] & 0  \ar[d] & \\ 0 \ar[r]  & N_{E/\mathbb{P}^{r+1}} \ar[r] \ar[d] & {N_{Y}}_{|E} \ar[r] \ar[d] & T^1_Y \ar[r] \ar[d]^{\cong} & 0 \\0 \ar[r] & \mathcal{O}_E(\Sigma)^{\oplus r+2-h} \ar[d] \ar[r] & \mathcal{M} \ar[r] \ar[d] & \mathcal{O}_{S} \ar[d] \ar[r] & 0 \\ & 0 & 0 & 0 &}
\end{equation}
Twisting the left vertical sequence
%$$0 \rightarrow N_{E/H^{\prime}} \rightarrow N_{E/\mathbb{P}^{r+1}} \rightarrow \mathcal{O}_E(\Sigma)^{\oplus r+2-h} \rightarrow 0$$
by $\mathcal{O}_E(-\Sigma)$ and using Lemma \ref{h1nd-delta}, one obtains $h^1(N_{E/\mathbb{P}^{r+1}}(-\Sigma))=r+2-h$. If $h=r+2$, the central horizontal sequence of diagram (\ref{ehprimeepr+1}) twisted by $\mathcal{O}_E(-\Sigma)$ then immediately gives $H^1({N_Y}_{|E}(-\Sigma))=(0)$. Assume then that $h \leq r+1$,
%Let $S := \left\{ p_1,...,p_h \right\}$, let $v_i$ be the tangent vector to $B$ at $p_i$.
%By assumption, there is a subscheme $T := \left\{p_1,...,p_{h-1}\right\} \subset S$ such that the set $\left\{v_1,...,v_{h-1} \right\}$ spans a general projective linear subspace of $\mathbb{P}^{r}$ of dimension $\min \left\{h-1,r\right\}$.
let $\widetilde{S} := \left\{(p_i,v_i)\right\}_{i=1,...,h} \subset \mathbb{P}(N_{E/\mathbb{P}^{r+1}}(-\Sigma))$, let $T$ be the projection of $K$ over $E$ and consider the diagram
\begin{equation}
\label{ef}
\xymatrix{ & 0 \ar[d] & 0 \ar[d] & 0 \ar[d] &\\ 0 \ar[r] & N_{E/\mathbb{P}^{r+1}}(-\Sigma) \ar[r] \ar[d]^{\cong} & \mathcal{E} \ar[r] \ar[d] & \mathcal{O}_T  \ar[d] \ar[r] & 0 \\ 0 \ar[r]  & N_{E/\mathbb{P}^{r+1}}(-\Sigma) \ar[r] \ar[d] & \mathcal{F} \ar[r] \ar[d] & \mathcal{O}_S \ar[r] \ar[d] & 0 \\ & 0 \ar[r] & \mathcal{O}_{S \smallsetminus T} \ar[r] \ar[d] & \mathcal{O}_{S \smallsetminus T} \ar[d] \ar[r] & 0 \\ & & 0 & 0 & }
\end{equation}
whose horizontal short exact sequences are the positive elementary transformations of $N_{E/\mathbb{P}^{r+1}}(-\Sigma)$ associated to $K$ and $\widetilde{S}$, respectively.\\
By Proposition \ref{propertieselementary} and Proposition \ref{isonormalelementary}, ${N_Y}_{|E}(-\Sigma)$ is isomorphic to $\mathcal{F}$.
%Since the points of $S$ are in linearly general position as points in $H \cap H^{\prime}$, the points of $T$ are in linearly general position too, hence projectively equivalent to $h-1$ general points of $E$. Moreover, the tangent vectors to $B$ at the points of $T$ span a general linear subspace of $\mathbb{P}^r$ of maximal dimension, thus the points of $\widetilde{T}$ are general as points of $\mathbb{P}(N_{E/\mathbb{P}^{r+1}}(-\Sigma))$
By assumptions and Proposition \ref{generalelementary}, one has $h^{1}(\mathcal{E})=\max \left\{0,h^{1}(N_{E/\mathbb{P}^{r+1}}(-\Sigma))-k \right\}=0$.\\
The central vertical sequence of (\ref{ef}) then gives $h^1({N_Y}_{|E}(-\Sigma))=0$.
\end{proof}
\begin{oss}
\label{h-1lgp}
Let $q_1,...,q_h$ be $h$ general points on a smooth elliptic curve. The divisor $M:=q_1+...+q_h$ gives an embedding of the curve as an elliptic normal curve $E \subset \mathbb{P}^{h-1}$, having $M$ itself as a hyperplane section. As a consequence, $p_1,...,p_h$ can always be assumed to be general points on $E$.
\end{oss}
\end{section}
\begin{section}{The main result}
We can now proceed with the proof of the main result, which goes on by induction on $r$. We start with a smooth half-canonical (nondegenerate) curve $C \subset \mathbb{P}^{r}$ of genus $g(r) := \lfloor \frac{r^2+10r+1}{4} \rfloor$ such that $C$ satisfies the following properties:
$$(*) \quad \left\{\begin{array}{ll}
               (i) & C \hbox{ is linearly normal i.e. } h^1(\mathcal{I}_{C}(1))=0; \\
               (ii) & C \hbox{ is 2-normal i.e. } h^1(\mathcal{I}_{C}(2))=0; \\
               (iii) & h^1(\mathcal{O}_{C}(2))=1; \\
               (iv) &  h^1(\mathcal{I}^2_{C}(2))=0.
             \end{array}\right.$$
In particular, by Fact \ref{ineq} the pair $(C,\mathcal{O}_{C}(1))$ is parameterized by an irreducible component $V_r \subset \mathcal{S}^r_{g(r)}$ having expected codimension.\\
We construct a suitable reducible curve $X := C \cup E \subset \mathbb{P}^{r+1}$, where $E$ is an elliptic normal curve whose degree is the right one to obtain $p_a(X)=g(r+1)$ and $\deg(X)=p_a(X)-1$. We show that properties $(*)$ ``propagate" to $X$ and that $X$ is smoothable to a curve $\Gamma \subset \mathbb{P}^{r+1}$. As a consequence of the argument outlined in the last part of the Introduction, $\Gamma$ will turn out to be a half-canonical curve still satisfying properties $(*)$, hence again Fact \ref{ineq} gives that the pair $(\Gamma,\mathcal{O}_{\Gamma}(1))$ is parameterized by an irreducible component $V_{r+1} \subset \mathcal{S}^{r+1}_{g(r+1)}$ having expected codimension.\\
\begin{teo}
\label{properties}
Fix integers $r \geq 2$ and $\lfloor \frac{r}{2} \rfloor+2 \leq h \leq r+2$. Let $C \subset H \cong \mathbb{P}^r \subset \mathbb{P}^{r+1}$ be a smooth (nondegenerate) half-canonical curve of genus $g \geq  \frac{r^2+r+2}{4}$ satisfying properties $(*)$ and such that $[C]$ is a general point of the irreducible component $W \subset \emph{Hilb}^r_{g,g-1}$ containing it. Assume that $C$ satisfies property $(\star \star)_r$ of Proposition \ref{r2secant}, let $S=S_h := \left\{p_1,...,p_h\right\}$ and
%il "the", i.e. il fatto che tale componente sia unica, è una conseguenza della Proposition 2.7.
let $H^{\prime} \subset \mathbb{P}^{r+1}$ be a $(h-1)$-plane cutting on $H$ the $(h-2)$-plane $<S_h> := H_{h-2}$.
%and thus $h$-secant to $C$ in the smooth $0$-dimensional subscheme $S=S_h=\left\{p_1,...,p_h\right\}$, falso per h=r+2.
Let $E=E_h \subset H^{\prime}$ be an elliptic normal curve intersecting $C$ transversally in $S_h$. Let $\Sigma$ be a hyperplane section of $E$ and let $X=X_h := C \cup E_h \subset \mathbb{P}^{r+1}$. Then $X$ satisfies properties $(*)$ and $h^1({N_X}_{|E}(-\Sigma))=0$.
\end{teo}
%L'assunzione sufficiently general su E non ha ragione di esistere. Il Lemma 3.5 non la chiede, la lineare normalità richiesta in (ii) è già chiesta nell'enunciato e l'annullarsi di h^2(\mathcal{I}_E(2)) vale sempre.
\begin{proof}
%The proof is by induction on $r$. The base case is given by \cite{BE}, Proposition 2.2. Let $C \subset H \cong \mathbb{P}^r$ be a smooth half-canonical (nondegenerate) curve $C \subset \mathbb{P}^{r}$ of genus $g(r)$ satisfying properties $(*)$.\\
%Embed $H$ as a hyperplane in $\mathbb{P}^{r+1}$, let $h=h(r)=\lfloor \frac{r}{2} \rfloor +2$ and let $H^{\prime}$ be a $h-1$-plane $h$-secant $C$ and not contained in $H$. Let $S$ be the smooth $0$-dimensional subscheme of length $h$ cut out on $C$ by $H^{\prime}$. Let $E \subset H^{\prime}$ be a sufficiently general elliptic normal curve (of degree $h$) intersecting $C$ transversally in $S$ and let $X := C \cup E$.
%One has $p_a(X)=g(r)+h=g(r+1)$ and $\deg X = \deg C + \deg E=g(r)-1+h=p_a(X)-1$.\\
Through all the proof, we will denote the ideal sheaf $\mathcal{I}_{S/\mathbb{P}^{r+1}}$ by $\mathcal{I}_S$.
\begin{itemize}
\item[$(i)$] It is sufficient to compute cohomology of the Mayer-Vietoris sequence $$0 \rightarrow \mathcal{I}_X(1) \rightarrow \mathcal{I}_{C/\mathbb{P}^{r+1}}(1) \oplus \mathcal{I}_{E/\mathbb{P}^{r+1}}(1) \rightarrow \mathcal{I}_{S}(1) \rightarrow 0$$  and note that $h^0(\mathcal{I}_X(1))=0$ since $X \subset \mathbb{P}^{r+1}$ is nondegenerate;
%i punti di S impongono il numero "giusto" di condizioni agli iperpiani di P^{r+1}, se $h \leq r+1$ per costruzione, se $h=r+2$ perché la curva ellittica è comunque sufficientemente generale.
\item[$(ii)$] Consider the exact sequence
$$0 \rightarrow \mathcal{I}_{X}(2) \rightarrow \mathcal{I}_{C/\mathbb{P}^{r+1}}(2) \xrightarrow{\alpha} \mathcal{I}_{C/X}(2) \rightarrow 0.$$
From $(ii)$ for $C$ it follows that $h^1(\mathcal{I}_{C/\mathbb{P}^{r+1}}(2))=0$,
%C'è il passaggio da P^r a P^{r+1} ma è la solita sequenza degli I
hence to prove $(ii)$ for $X$ it is sufficient to show that the map $H^{0}(\alpha)$ is surjective. One has $\mathcal{I}_{C/X}(2) \cong \mathcal{O}_E(2\Sigma-S) \cong \mathcal{O}_E(\Sigma)$ since $S$ is a hyperplane section of $E$. On the other hand, $C$ is contained in $H$, thus $\mathcal{I}_{C/\mathbb{P}^{r+1}}(2)$ possesses global sections whose zero locus is a reducible hyperquadric split in two hyperplanes one of which is $H$. As a consequence, $H^{0}(\mathcal{I}_{C/\mathbb{P}^{r+1}}(2)) \supset H^{0}(\mathcal{O}_{\mathbb{P}^{r+1}}(1))$. Since $E$ is linearly normal, $H^{0}(\alpha)$ is surjective;
\item[$(iii)$]
The cohomology sequence of
\begin{equation}
\label{ix2opr+12}
0 \rightarrow \mathcal{I}_X(2) \rightarrow \mathcal{O}_{\mathbb{P}^{r+1}}(2) \rightarrow \mathcal{O}_X(2) \rightarrow 0
\end{equation}
gives $h^{1}(\mathcal{O}_X(2))=h^{2}(\mathcal{I}_X(2))$. Consider the Mayer-Vietoris sequence
\begin{equation}
\label{mvx}
0 \rightarrow \mathcal{I}_X(2) \rightarrow \mathcal{I}_{C/\mathbb{P}^{r+1}}(2) \oplus \mathcal{I}_{E/\mathbb{P}^{r+1}}(2) \rightarrow \mathcal{I}_{S}(2) \rightarrow 0.
\end{equation}
Since the points of $S$ impose independent conditions to hyperquadrics in $\mathbb{P}^{r+1}$, the exact sequence
$$0 \rightarrow \mathcal{I}_{S}(2) \rightarrow \mathcal{O}_{\mathbb{P}^{r+1}}(2) \rightarrow \mathcal{O}_{S}(2) \rightarrow 0$$
gives $h^{1}(\mathcal{I}_{S}(2))=h^{2}(\mathcal{I}_{S}(2))=0$.\\
Sequences analogous to (\ref{ix2opr+12}) for $C$ and $E$ give $h^{2}(\mathcal{I}_{C/\mathbb{P}^{r+1}}(2))=h^{1}(\mathcal{O}_C(2))=1$ and $h^2(\mathcal{I}_{E/\mathbb{P}^{r+1}}(2))=0$, hence (\ref{mvx}) gives $h^{1}(\mathcal{O}_X(2))=1$;
\item[$(iv)$] Since $C$ is linearly normal, by Lemma \ref{bdegenerate} one has $h^{1}(\mathcal{I}^2_{C/\mathbb{P}^{r+1}}(2))=0$. By Lemma \ref{i2}, there exists an exact sequence
\begin{equation}
\label{i2x2i2c2}
0 \rightarrow \mathcal{I}^2_{X}(2) \rightarrow \mathcal{I}^2_{C/\mathbb{P}^{r+1}}(2) \rightarrow \mathcal{A} \rightarrow 0
\end{equation}
where
\begin{equation}
\label{bai2cx2}
0 \rightarrow \mathcal{B} \rightarrow \mathcal{A} \rightarrow \mathcal{I}^2_{C/X}(2) \rightarrow 0
\end{equation}
with $\mathcal{B} \cong \mathcal{I}_{C/X} \otimes N^{\vee}_{X}(2) \cong {N^{\vee}_{X}}_{|E}(2\Sigma-S) \cong {N^{\vee}_{X}}_{|E}(\Sigma)$ and $\mathcal{I}^2_{C/X}(2) \cong \mathcal{O}_E(2\Sigma-2S) \cong \mathcal{O}_E$.\\
Serre duality gives $h^{0}({N^{\vee}_{X}}_{|E}(\Sigma))=h^{1}({N_{X}}_{|E}(-\Sigma))$. Let $l_i$, $i=1,...,h$, be the tangent line to $C$ at the point $p_i$ and let $v_i=\mathbb{P}(l_i)$. Let $W(-S) \subset W$ be the irreducible component of the subscheme parameterizing curves passing through $S$ which contains $[C]$.
%Ce n'è una sola perché C è general come punto di W.
Imposing to a curve parameterized by $W$ the passage through $S$ accounts for at most $h(r-1)$ conditions, hence $ \dim W(-S) \geq \dim W-h(r-1)$.
%Penso che W(-S) sia irriducibile, ma non mi serve dimostrarlo.
Imposing to a curve of $W(-S)$ a fixed tangency condition at $k=r+2-h$ points of $S$ accounts for at most $k(r-1)$ additional conditions. By the assumptions on $g$, one has $\dim W(-S)-k(r-1) \geq \dim W-(h+k)(r-1) \geq \chi(N_{C/\mathbb{P}^r})-(r+2)(r-1)=4(g-1)-(r+2)(r-1) \geq 0$. The fact that $[C]$ is a general point of $W$, combined with Remark \ref{h-1lgp}, then gives that, for $h \leq r+1$,
%the tangent vectors to $C$ at the points of $S$ span a general linear projective subspace of $H$ of maximal dimension, thus Lemma \ref{nye} applies and $h^{1}({N_{X}}_{|E}(-\Sigma))=0$.
%Since C is smooth and W(-S) is irreducible, a deformation of C in W(-S) having general tangency conditions at k points of S must be smooth too. In questo modo abbiamo scongiurato il caso che C liscia appartenga a una componente irriducibile delle curve di W passanti per S, e C' che ha tangenze generali appartenga ad un'altra componente irriducibile, ma non sia liscia.
the 0-dimensional scheme $K := \left\{(p_i,v_i)\right\}_{i=1,...,k} \subset \mathbb{P}(N_{E/\mathbb{P}^{r+1}}(-\Sigma))$ is general, hence Lemma \ref{nye} applies and $h^{1}({N_{X}}_{|E}(-\Sigma))=0$.
The cohomology sequence of (\ref{bai2cx2}) then gives $h^0(\mathcal{A}) \leq 1$. Since $C$ is a nondegenerate curve in $H \cong \mathbb{P}^r$, one has $h^0(\mathcal{I}^2_{C/\mathbb{P}^{r+1}}(2))=1$ (the only hyperquadric whose singular locus contains $C$ is $2H$), hence the cohomology sequence of (\ref{i2x2i2c2}) gives $h^0(\mathcal{A})=1$ and $h^1(\mathcal{I}^2_{X}(2))=0$.
\end{itemize}
\end{proof}
\begin{prop}
\label{assumexsmoothable}
Notation as in Theorem \ref{properties}, assume that $X=X_h$ is smoothable. Then, it is smoothable to a half-canonical curve $\Gamma \subset \mathbb{P}^{r+1}$ satisfying properties $(*)$. In particular, the Gaussian map $\Psi_{\mathcal{O}_{\Gamma}(1)}$ is injective.
\begin{proof}
Let $\Gamma \subset \mathbb{P}^{r+1}$ be a general smoothing of $X$. One has $g(\Gamma)=p_a(X)=g(C)+h$ and $\deg(\Gamma)=\deg(C)+\deg(E_h)=g(C)-1+h=g(\Gamma)-1$. Since $\chi(\mathcal{I}_{\Gamma}(2))=\chi(\mathcal{I}_X(2))$ and $h^1(\mathcal{I}_X(2))=0$, sequence (\ref{ix2opr+12}) for $X$ and $\Gamma$ and the upper semicontinuity of the cohomology give $1=h^2(\mathcal{I}_X(2))=h^2(\mathcal{I}_{\Gamma}(2))=h^{1}(\mathcal{O}_{\Gamma}(2))$, hence $\mathcal{O}_{\Gamma}(2) \cong \omega_{\Gamma}$. Properties $(i)$ and $(iv)$ of $(*)$ follow from the upper semicontinuity of the cohomology too.
By Proposition \ref{h1i2c}, the Gaussian map $\Psi_{\mathcal{O}_{\Gamma}(1)}$ is injective.
%By Proposition \ref{h1i2c} and Corollary \ref{cornag}, the point $(\Gamma,\mathcal{O}_{\Gamma}(1))$ is a smooth point of an irreducible component $V_{r+1} \subset \mathcal{S}^{r+1}_{g(r+1)}$ having expected codimension in $\mathcal{S}_{g(r+1)}$, and the inductive step is proved.
\end{proof}
\end{prop}
\begin{prop}
\label{dimwrgh}
Notation as in Theorem \ref{properties}, let $W^{r+1}_{p_a(X),h} \subset \emph{Hilb}^{r+1}_{p_a(X),\deg(X)}$ be the (closure of the) locus parameterizing curves $X:=X_h$. Then $W^{r+1}_{p_a(X),h}$ is equidimensional of dimension $3p_a(X)-5+{r+3 \choose 2}$. In particular, if $X_h$ is smoothable to a curve $\Gamma$, then $W^{r+1}_{p_a(X),h}$ has codimension 1 inside $\emph{Hilb}^{r+1}_{p_a(X),\deg(X)}$.
\begin{proof}
We want to compute the number $n_h$ of parameters a curve $X_h:=C \cup E_h$ parameterized by an irreducible component of $W^{r+1}_{p_a(X),h}$ depends on.
Let $\mathcal{E} \subset \text{Hilb}^{h-1}_{1,h}$ be the subscheme parameterizing smooth elliptic normal curves containing $S_h$, and let $R_h \subset \mathcal{E}$ be the irreducible component containing $[E_h]$.
By Lemma \ref{h1nd-delta}, the normal bundle $N_{E_h/H^{\prime}}$ satisfies $h^1(N_{E_h/H^{\prime}}(-S_h))=0$.
Thus \cite[Lemma 2.4]{BABEF} assures that $R_h$ is smooth at $[E_h]$, and $\dim R_h=\dim_{[E_h]}\text{Hilb}^{h-1}_{1,h}-h(h-2)=h^2-(h^2-2h)=2h$.
Using Proposition \ref{assumexsmoothable} and Proposition \ref{hilbertscheme} one then obtains
\begin{align*}
n_h    &  =\underbrace{3g(C)-4+{r+2 \choose 2}}_{\dim W} + \underbrace{r+1}_{\textrm{choice of } H \subset \mathbb{P}^{r+1}} + \underbrace{2h-r-2}_{\dim V_h} + \underbrace{r+2-h}_{\textrm{choice of }H^{\prime} \supset <S_h>}+ \underbrace{2h}_{\dim R_h}=\\
 &  =3(g(C)+h)-4+{r+3 \choose 2}-1=\dim_{[\Gamma]} \text{Hilb}^{r+1}_{p_a(X),\deg(X)}-1
\end{align*}
\end{proof}
\end{prop}

The most difficult part of our argument is proving the smoothability of the curves $X:=X_h \subset \mathbb{P}^{r+1}$. Since we mostly deal with curves $X$ such that $h^1(N_X)>>0$ and $\chi(N_X)$ is less than the number of parameters counting locally trivial deformations of $X$, standard techniques fail to provide a proof in our situation. Nevertheless, exploiting the very specific fact that $\mathcal{O}_X(\omega_X-2H)$ has degree 0 on $X$, one can obtain a lower bound on $h^0(N_X)$ which, under delicate additional conditions, guarantees the existence of ``enough" embedded deformations for $X$.
\begin{lem}
\label{h0nx}
Notation as in Theorem \ref{properties}, let $\lfloor \frac{r}{2} \rfloor +2 \leq h \leq r+2$ and let $X_h:=C \cup E_h \subset \mathbb{P}^{r+1}$. Then one has $h^0(N_{X_h})=3p_a(X_h)-4+{r+3 \choose 2}$. If $h \geq \lfloor \frac{r}{2} \rfloor +3$, let $q_{h-1}$ be a general point of $E_{h-1}$, $p^{\prime}_h$ be a general point of $C$, $l:=<q_{h-1},p^{\prime}_h>$, $m:=<p_{h-1},p^{\prime}_{h}>$ and $Z:=X_{h-1} \cup m$. Then one has $h^0(N_{X_{h-1} \cup l})=3p_a(X_h)-4+{r+3 \choose 2}$. Moreover, if $g(C) \leq {r+2 \choose 2}-1$, $h^0(N_Z)=3p_a(X_h)-4+{r+3 \choose 2}$.
\end{lem}
\begin{proof}
Let $X:=X_h$, $E:=E_h$ and $S:=S_h$. One has $\mathcal{O}_X(\omega_X-2H)_{|C}=\mathcal{O}_C(S)$ and $\mathcal{O}_X(\omega_X-2H)_{|E}=\mathcal{O}_E(-S)$, hence, twisting the short exact sequence
$$0 \rightarrow \mathcal{O}_C(-S) \rightarrow \mathcal{O}_X \rightarrow \mathcal{O}_E \rightarrow 0$$
by $N_X(\omega_X-2H)$ one obtains the exact sequence
$$0 \rightarrow {N_X}_{|C} \rightarrow N_{X}(\omega_X-2H) \rightarrow {N_X}_{|E}(-S) \rightarrow 0.$$
One the other hand, (\ref{icnxn}) writes as
\begin{equation}
\label{nxes}
0 \rightarrow {N_X}_{|E}(-S) \rightarrow N_X \rightarrow {N_X}_{|C} \rightarrow 0.
\end{equation}
Computing cohomology of both sequences and using Theorem \ref{properties} one obtains $h^{0}(N_X) \geq h^{0}(N_X(\omega_X-2H))$.\\
Let's compute $h^0(N_X(\omega_X-2H))$. Serre duality gives $h^0(N_X(\omega_X-2H))=h^1(N^{\vee}_X(2))$.
%Consider the short exact sequence
%\begin{equation}
%\label{i2x2ix2}
%0 \rightarrow \mathcal{I}^2_{X}(2) \rightarrow \mathcal{I}_X(2) \rightarrow N^{\vee}_X(2) \rightarrow 0.
%\end{equation}
Since $X$ is a locally complete intersection curve, one has $4(p_a(X)-1)=\chi(N_X)=\chi(N_X(\omega_X-2H))=-\chi(N^{\vee}_X(2))$. Since $X \subset \mathbb{P}^{r+1}$ is nondegenerate and $h^1(\mathcal{I}_X(2))=h^1(\mathcal{I}^2_X(2))=0$ by properties $(ii)$ and $(iv)$ of Theorem \ref{properties}, computing cohomology of sequence (\ref{eq1}) yields $h^0(N^{\vee}_X(2))=h^0(\mathcal{I}_X(2))=h^0(\mathcal{O}_{\mathbb{P}^{r+1}}(2))-h^0(\mathcal{O}_X(2))={r+3 \choose 2}-p_a(X)$, from which $h^1(N^{\vee}_X(2))=3p_a(X)-4+{r+3 \choose 2}$.\\
%l.c.i. serve come ipotesi per \chi(N_X)
%dimostrazione alternativa.
%Since $X$ is a locally complete intersection curve, one has $4(p_a(X)-1)=\chi(N_X)=\chi(N_X(\omega_X-2H))=-\chi(N^{\vee}_X(2))$. Using sequence (\ref{i2x2ix2}) one gets $\chi(\mathcal{I}^2_X(2))={r+3 \choose 2}-p_a(X)+1+4(p_a(X)-1)$.
%l.c.i. serve come ipotesi per \chi(N_X)
%Property $(ii)$ of Theorem \ref{properties} yields the associated cohomology sequence
%$$0 \rightarrow H^1(N^{\vee}_X(2)) \rightarrow H^2(\mathcal{I}^2_X(2)) \rightarrow H^2(\mathcal{I}_X(2)) \rightarrow 0.$$
%One has $h^2(\mathcal{I}_X(2))=h^1(\mathcal{O}_X(2))=1$ (see the proof of $(iii)$ of Theorem \ref{properties}) and $\chi(\mathcal{I}^2_X(2))=h^2(\mathcal{I}^2_X(2))$ by $(iv)$ of Theorem \ref{properties}, hence $h^0(N_X) \geq h^0(N_X(\omega_X-2H))=h^1(N^{\vee}_X(2))=\chi(\mathcal{I}^2_X(2))-h^2(\mathcal{I}_X(2))=3p_a(X)-4+{r+3 \choose 2}$.\\
Now, consider the commutative exact diagram
\begin{equation}
\label{cxh}
\xymatrix{ & 0 \ar[d] & 0 \ar[d] & &\\ 0 \ar[r] & N_{C/H} \ar[r]^{\cong} \ar[d] & N_{C/H} \ar[r] \ar[d] & 0  \ar[d] & \\ 0 \ar[r]  & N_{C/\mathbb{P}^{r+1}} \ar[r] \ar[d] & {N_{X}}_{|C} \ar[r] \ar[d] & T^1_{X} \ar[r] \ar[d]^{\cong} & 0 \\0 \ar[r] & \mathcal{O}_C(\Sigma) \ar[d] \ar[r] & \mathcal{M} \ar[r] \ar[d] & T^1_{X} \ar[d] \ar[r] & 0 \\ & 0 & 0 & 0 &}
\end{equation}
where $\Sigma$ is a hyperplane section of $C$. An argument analogous to the one carried out in \cite[Lemma 3.2]{BE}, gives $\mathcal{M} \cong \mathcal{O}_C(\Sigma+S)$. Serre duality gives $h^0(\mathcal{O}_C(\Sigma+S))=h^1(\mathcal{O}_C(\Sigma-S))$. Since $<S> \cong \mathbb{P}^{h-2}$, one has $h^0(\mathcal{O}_C(\Sigma-S))=r+1-(h-1)=r+2-h$, hence Riemann-Roch theorem gives $h^1(\mathcal{O}_C(\Sigma-S))=-g(C)+1+h-1+g(C)+r+2-h=r+2$.
%and $S_{h+1}$
%are not in linearly general position as points in $\mathbb{P}^r$ and since $C$ is half-canonically embedded, then these points are not in linearly general position in the canonical embedding of $C$ too. As a consequence, geometric Riemann-Roch and $h \leq r+1$ give $h^1(\mathcal{O}_C(\Sigma+S_h)) \geq h^1(\mathcal{O}_C(\Sigma))-(h-1) \geq 1$.\\
%and $h^1(\mathcal{O}_C(\Sigma+S_{h+1})) \geq h^1(\mathcal{O}_C(\Sigma))-h$, respectively.\\
Taking cohomology of the central vertical sequence of diagram (\ref{cxh}) and using Proposition \ref{hilbertscheme}, one obtains $h^0({N_X}_{|C}) \leq 3g(C)-4+{r+2 \choose 2} + r+2$. From Theorem \ref{properties} one gets $h^0({N_X}_{|E}(-S))=3h$, hence taking cohomology of (\ref{nxes}) gives $h^0(N_X) \leq 3p_a(X)-4+{r+3 \choose 2}$.\\
Let now $E:=E_{h-1}$ and $S:=S_{h-1}$. The central vertical sequence of diagram (\ref{cxh}) gives $h^1({N_{X_{h-1}}}_{|C}) \geq 1$. Since by Lemma \ref{ehlsmoothable} and $(\star \star)_r$ $(ii)$ of Proposition \ref{r2secant} $X_{h-1} \cup l$ deforms to $X_h$, by the upper semicontinuity of the cohomology it is sufficient to show that $h^0(N_{X_{h-1} \cup l}) \leq 3p_a(X_h)-4+{r+3 \choose 2}$. Consider the exact sequence (see (\ref{icnxn}))
\begin{equation}
\label{nxh1l}
0 \rightarrow {N_{X_{h-1} \cup l}}_{|l}(-q_{h-1}-p^{\prime}_h) \rightarrow N_{X_{h-1} \cup l} \rightarrow {N_{X_{h-1} \cup l}}_{|X_{h-1}} \rightarrow 0
\end{equation}
Seeing ${N_{X_{h-1} \cup l}}_{|l}$ as a positive elementary transformation of $N_l \cong \mathcal{O}_{\mathbb{P}^1}(1)^{\oplus r}$ it is immediate to see that $h^0({N_{X_{h-1} \cup l}}_{|l}(-q_{h-1}-p^{\prime}_h))=2$. Let $b$ be the tangent line to $l$ at $p^{\prime}_h$ and let $v=\mathbb{P}(b)$. Up to deforming $C$ keeping the points of $S_h$ fixed, the point $(p^{\prime}_{h},v)$ is general as a point of $\mathbb{P}({N_{X_{h-1}}}_{|C})$, hence Proposition \ref{generalelementary} and Proposition \ref{isonormalelementary} give
%lasciando fissi p_1,...,p_h e E_h (e quindi v_1,...,v_h) posso ottenere qualunque v_{h+1}
\begin{equation}
\label{nxhlcnxhc}
h^1({N_{X_{h-1} \cup l}}_{|C}) = h^1({N_{X_{h-1}}}_{|C})-1.
\end{equation}
Twisting the exact sequence $0 \rightarrow \mathcal{O}_E(-S) \rightarrow \mathcal{O}_{X_{h-1}} \rightarrow \mathcal{O}_C \rightarrow 0$ by $N_{X_{h-1}}$ and using Theorem \ref{properties} one obtains $h^1(N_{X_{h-1}})=h^1({N_{X_{h-1}}}_{|C})$. Seeing ${N_{X_{h-1} \cup l}}_{|E}$ as a positive elementary transformation of ${N_{X_{h-1}}}_{|E}$ and using again Theorem \ref{properties} it is immediate to see that $h^1({N_{X_{h-1} \cup l}}_{|E}(-S))=0$, thus twisting the above sequence by ${N_{X_{h-1} \cup l}}_{|X_{h-1}}$ gives
%vedi p. -25 degli appunti ${N_{X_h \cup l}}_{|X_h} \rightarrow {N_{X_h \cup l}}_{|C}$ is supported on a curve
$h^1({N_{X_{h-1} \cup l}}_{|X_{h-1}}) = h^1({N_{X_{h-1} \cup l}}_{|C})$. Equality (\ref{nxhlcnxhc}) then gives $h^0(N_{X_{h-1}}) = h^{0}({N_{X_{h-1} \cup l}}_{|X_{h-1}})-1$. Sequence (\ref{nxh1l}) then yields $h^0(N_{X_{h-1} \cup l}) \leq 2+ h^0(N_{X_{h-1}})+1=h^0(N_{X_h})$.\\
If $g(C) \leq {r+2 \choose 2}-1$, by Proposition \ref{hilbertscheme} one has that $h^1(N_{C/H}) \geq 1$. By construction $Z$ deforms to $X_{h-1} \cup l$, hence by the upper semicontinuity of the cohomology it is sufficient to show that $h^0(N_Z) \leq 3p_a(X_{h})-4+{r+3 \choose 2}$.
By \cite[(D.2) and the proof of Lemma D.1.3 $(ii)$]{SE} there exists a short exact sequence
\begin{equation}
\label{nveez}
0 \rightarrow {N^{\vee}_{Z}}_{|C \cup m} \rightarrow N^{\vee}_{C \cup m/\mathbb{P}^{r+1}} \rightarrow N^{\vee}_{C \cup m/Z} \rightarrow 0
\end{equation}
where $N^{\vee}_{C \cup m/Z} \cong \mathcal{I}_{C \cup m /Z}/\mathcal{I}^2_{C \cup m/Z} \cong \mathcal{O}_E(-S)/\mathcal{O}_E(-2S) \cong \mathcal{O}_S$.
Dualizing (\ref{nveez}) one obtains the short exact sequence
$$0 \rightarrow N_{C \cup m/\mathbb{P}^{r+1}} \rightarrow {N_{Z}}_{|C \cup m} \rightarrow \mathcal{O}_S \rightarrow 0$$
which fits into the commutative exact diagram
\begin{equation}
\label{czh}
\xymatrix{ & 0 \ar[d] & 0 \ar[d] & &\\ 0 \ar[r] & N_{C \cup m/H} \ar[r]^{\cong} \ar[d] & N_{C \cup m/H} \ar[r] \ar[d] & 0  \ar[d] & \\ 0 \ar[r]  & N_{C \cup m/\mathbb{P}^{r+1}} \ar[r] \ar[d] & {N_{Z}}_{|C \cup m} \ar[r] \ar[d] & \mathcal{O}_S \ar[r] \ar[d]^{\cong} & 0 \\0 \ar[r] & \mathcal{O}_{C \cup m}(\Sigma) \ar[d] \ar[r] & \mathcal{L} \ar[r] \ar[d] & \mathcal{O}_S \ar[d] \ar[r] & 0 \\ & 0 & 0 & 0 &}
\end{equation}
An argument analogous to the one used in \cite[Lemma 3.2]{BE} shows that $\mathcal{L}$ is a line bundle.
Twisting the short exact sequence $0 \rightarrow \mathcal{O}_m(-p-q) \rightarrow \mathcal{O}_{C \cup m} \rightarrow \mathcal{O}_C \rightarrow 0$ by $\mathcal{L}$ one obtains
\begin{equation}
\label{lmpq}
0 \rightarrow \mathcal{L}_{|m}(-p-q) \rightarrow \mathcal{L} \rightarrow \mathcal{L}_{|C} \rightarrow 0.
\end{equation}
Note that the restriction of last horizontal sequence of diagram (\ref{czh}) to $C$ is exactly the last horizontal sequence of diagram (\ref{cxh}) (written down for $X=X_{h-1}$), hence $\mathcal{L}_{|C} \cong \mathcal{O}_C(\Sigma+S)$. On the other hand, since by Riemann-Roch theorem one has $\chi(\mathcal{O}_{C}(\Sigma))=\chi(\mathcal{O}_{C \cup m}(\Sigma))=0$, it follows that $\chi(\mathcal{L}_{|m}(-p-q))=\chi(\mathcal{L}) - \chi(\mathcal{L}_{|C})=h-h=0$, from which one has that $\mathcal{L}_{|m}(-p-q) \cong \mathcal{O}_{\mathbb{P}^1}(-1)$.
Computing cohomology of sequence (\ref{lmpq}) then gives $h^{0}(\mathcal{L})=r+2$. Let us compute $h^0(N_{C \cup m/H})$. Let $b$ be the tangent line to $m$ at $p^{\prime}_{h}$ and let $v=\mathbb{P}(b)$. Up to deforming $C$ keeping $p_{h-1}$ fixed, we can assume that $(p^{\prime}_{h},v) \in \mathbb{P}(N_{C/H})$ is a general point, hence by Proposition \ref{generalelementary} the elementary transformation $\mathcal{F}$ of $N_{C/H}$ associated to $(p^{\prime}_{h},v)$ has $h^1(\mathcal{F})=h^1(N_{C/H})-1$.
The existence of an exact sequence $0 \rightarrow \mathcal{F} \rightarrow {N_{C \cup m/H}}_{|C} \rightarrow \mathcal{O}_{p_{h-1}} \rightarrow 0$ then gives $h^0({N_{C \cup m/H}}_{|C}) \leq h^0(N_{C/H})+1$. Using Proposition \ref{hilbertscheme} and computing cohomology of the exact sequence
\begin{equation}
\label{ncmh}
0 \rightarrow {N_{C \cup m/H}}_{|m}(-p-q) \rightarrow N_{C \cup m/H} \rightarrow {N_{C \cup m/H}}_{|C} \rightarrow 0
\end{equation}
one then obtains $h^0(N_{C \cup m/H}) \leq 2+3g(C)-4+{r+2 \choose 2}+1=3g(C)-1+{r+2 \choose 2}$, thus the central vertical sequence of diagram (\ref{czh}) gives $h^0({N_Z}_{|C \cup m}) \leq 3g(C)-1+{r+2 \choose 2} + r+2=3g(C)-1+{r+3 \choose 2}$.
%ricordati che p^\prime}_{h+1} è già un punto generale di C.
In the end, consider the exact sequence $0 \rightarrow {N_Z}_{|E}(-S) \rightarrow N_Z \rightarrow {N_Z}_{|C \cup m} \rightarrow 0$. Again by \cite[(D.2) and the proof of Lemma D.1.3 $(ii)$]{SE} there exists a short exact sequence $0 \rightarrow {N^{\vee}_Z}_{|X_{h-1}} \rightarrow N^{\vee}_{X_{h-1}} \rightarrow \mathcal{I}_{X_{h-1}/Z}/\mathcal{I}^2_{X_{h-1}/Z} \rightarrow 0$, with $\mathcal{I}_{X_{h-1}/Z}/\mathcal{I}^2_{X_{h-1}/Z} \cong \mathcal{O}_m(-p^{\prime}_h)/\mathcal{O}_m(-2p^{\prime}_h) \cong \mathcal{O}_{p^{\prime}_h}$. Twisting by $\mathcal{O}_E(S)$ and dualizing one obtains ${N_Z}_{|E}(-S) \cong {N_{X_{h-1}}}_{|E}(-S)$. By Theorem \ref{properties} one has $h^0({N_{X_{h-1}}}_{|E}(-S))=3(h-1)$, thus $h^0(N_Z) \leq  3(h-1)+3g(C)-1+{r+3 \choose 2}=3(g(C)+h)-4+{r+3 \choose 2}=3p_a(X_{h})-4+{r+3 \choose 2}$.
\end{proof}
\begin{lem}
\label{sameirreduciblecomponent}
Notation as in Lemma \ref{h0nx} assume that, for a fixed $\lfloor \frac{r}{2} \rfloor +3 \leq h \leq r+2$ there exist smoothings $C_{h-1}, C_h$ of curves $X_{h-1}, X_h$, respectively. Let $n$ be a 2-secant line to $C_{h-1}$. Then the curves $C_{h}, C_{h-1} \cup n, X_{h-1} \cup l$ and $X_h$ are all parameterized by smooth points of the same irreducible component $W^{r+1}_{p_a(X_h)} \subset \emph{Hilb}^{r+1}_{p_a(X_h),p_a(X_h)-1}$, which has dimension $3p_a(X_h)-4+{r+3 \choose 2}$.
\end{lem}
\begin{proof}
by Lemma \ref{ehlsmoothable} and $(\star \star)_r$ $(ii)$ of Proposition \ref{r2secant}, one has that $X_{h-1} \cup l$ deforms to $X_h$. As a consequence, there exists an open neighborhood $U \subset \text{Hilb}^{r+1}_{p_a(X_h),p_a(X_h)-1}$ of $[X_{h-1} \cup l]$ such that $[X_h], [C_{h-1} \cup n], [C_h] \in U$. It is then sufficient to show that $[X_{h-1} \cup l]$ is a smooth point of $\text{Hilb}^{r+1}_{p_a(X_h),p_a(X_h)-1}$, which immediately follows from Proposition \ref{dimwrgh}, the fact that $[X_{h-1} \cup l]$ belongs to $W^{r+1}_{p_a(X_h),h}$ and the fact that $T_{[X_{h-1} \cup l]}\text{Hilb}^{r+1}_{p_a(X_h),p_a(X_h)-1}=h^{0}(N_{X_{h-1} \cup l})=3p_a(X_h)-4+{r+3 \choose 2}$ by Lemma \ref{h0nx}.
%m è una secante generale a C_h, perché p_k e p^{\prime}_{k+1} possono essere assunti generali A CAUSA DEL FATTO CHE LA VARIETA' DELLE SECANTI HA DIMENSIONE POSITIVA. Ricordati pertanto, qualora dovessi abbassare il range da \lfloor \frac{r}{2} \rfloor +3 a \lfloor \frac{r}{2} \rfloor +1, che la generalità non vale più.
%aggiungere la retta è come aggiungere il dato di due punti sulla curva, quindi dal fatto che X_h si deforma a C_h segue ovviamente che X_h \cup l si deforma a C_h cup n.
\end{proof}
\begin{teo}
\label{maineq}
For all integers $r \geq 2$ and $g$ satisfying
\begin{equation}
\label{intervallog}
\lfloor \frac{r^2+10r+1}{4} \rfloor := g(r) \leq g \leq {r+2 \choose 2}
\end{equation}
%ricontrollato due volte, è giusto.
there exists a smooth (nondegenerate) half-canonical curve $\Gamma=\Gamma_{r,g} \subset \mathbb{P}^r$ satisfying properties $(*)$. In particular, the Gaussian map $\Psi_{\mathcal{O}_{\Gamma}(1)}$ is injective, and the pair $(\Gamma,\mathcal{O}_{\Gamma}(1))$ is parameterized by an irreducible component $\widetilde{V}_{r,g} \subset \mathcal{S}^r_{g}$ having expected codimension in $\mathcal{S}_{g}$.
\end{teo}
\begin{proof}
The proof is by induction on $r$. The base case ($r=2$) is given by Proposition \ref{r=2}, a quintic plane curve $C^6$ of genus $6$ such that $(C^6,\mathcal{O}_{C^6}(1))$ is a general point of $V=\widetilde{V}_{2,6}$. It is straightforward to verify that this curve satisfies the assumptions of Theorem \ref{properties}. Assume by induction that, for all $g(r) \leq g \leq {r+2 \choose 2}$, $C^g \subset H \cong \mathbb{P}^r$ is a smooth (nondegenerate) half-canonical curve of genus $g$ satisfying the assumptions of Theorem \ref{properties}.
%and that $(C^g,\mathcal{O}_{C^g}(1))$ is a general point of $\widetilde{V}_{r,g}$. Non serve. Questo fatto è una conseguenza del fatto che C^g soddisfa le (*), ma non entra nell'argomento induttivo.
In particular, the curve $C^g$ satisfies property $(\star \star)_r$ of Proposition \ref{r2secant}. Let then $X^g_h := C^g \cup E_h \subset \mathbb{P}^{r+1}$ be a curve constructed as in Theorem \ref{properties}: at every step of the induction process, one attaches to $C^g$ an elliptic normal curve intersecting it in $h$ points. By Theorem \ref{properties}, $X^g_h$ satisfies properties $(*)$.\\
For $\lfloor \frac{r}{2} \rfloor +3 \leq h \leq r+2$, the arithmetic genus of $X^g_h$ lies in the interval between
$$6+ \sum_{i=2,\, i \scriptsize{\hbox{ odd }}}^{r} \left(\frac{i+1}{2} +2\right) + \sum_{i=2,\, i \scriptsize{\hbox{ even }}}^{r} \left(\frac{i}{2}+3\right)=6+\sum_{i=2}^{r} \left(\frac{i}{2}+3\right)-\sum_{i=2, \, i \scriptsize{\hbox{ odd }}}^{r} \frac{1}{2}=$$
$$=\lfloor \frac{(r+1)^2+10(r+1)+1}{4} \rfloor:=g(r+1)$$
and
$$6+\sum_{i=2}^{r}\left(i+2\right)={(r+1)+2 \choose 2}.$$
We want to show that, for these values of $h$, $X^g_h$ is smoothable.
If $g={r+2 \choose 2}$ and $h=r+2$, $X^g_h$ is smoothable by \cite[Lemma 3.2]{BE} (it is the ``trivial" case, where $h^1(N_{X^g_h})=0$). In the other cases, suppose by contradiction that $X^g_h$ admits (as an embedded curve) only locally trivial deformations i.e. that, notation as in Proposition \ref{dimwrgh}, in a neighborhood of $[X^g_h]$ the Hilbert scheme $\text{Hilb}^{r+1}_{p_a(X^g_h),p_a(X^g_h)-1}$ is set-theoretically the locus $W^{r+1}_{p_a(X^g_h),h}$. We have to consider two cases:\\
\begin{itemize}
\item[{$\left[g \leq {r+2 \choose 2}-1\right]$}] Let $n$ be a 2-secant line to $C^g$. Let $Z=X^g_{h-1} \cup m$ and $X^g_{h-1} \cup l$ be as in Lemma \ref{h0nx}. By construction, Lemma \ref{ehlsmoothable} and $(\star \star)_r$ $(ii)$ of Proposition \ref{r2secant}, the curves $X^g_h , X^g_{h-1} \cup l$ and $Z$ are parameterized by the same irreducible component of $W^{r+1}_{p_a(X^g_h),h}$.
\item[{$\left[g={r+2 \choose 2}\right]$}] We have two subcases:\\
\begin{itemize}
%\item[{$\left[r=3\right]$}] There is only one curve to consider, namely $X^6_4$, which is smoothable by \cite[Lemma 3.2]{BE} (it is the ``trivial" case, where $h^1(N_{X^g_h})=0$).
%già esaminato nella frase sopra
\item[{$\left[r=3\right]$}] There is only one curve to consider, namely $X^{10}_4 \subset \mathbb{P}^4$, for which the proof goes on as in the case $r \geq 4$, provided that we define ``good" curves $C^9 \subset \mathbb{P}^3$ (which is not ``reached" by our induction argument, as its genus is too low) and $X^9_4 \subset \mathbb{P}^4$, and we prove that, if $n$ is a 2-secant line to $C^9$, one has that $C^9 \cup n$ and $C^{10}$ (which is obtained by smoothing $X^6_4 \subset \mathbb{P}^3$) are parameterized by the same irreducible component of $\text{Hilb}^{3}_{10,9}$.\\
    By Lemma \ref{h0nx}, one has $\dim_{[X^6_3]}\text{Hilb}^3_{9,8} \leq 33$. Then, by \cite[Proposition 4.7 and Theorem 3.7]{SA}, $X^6_3$ is smoothable to a complete intersection $C^9$ of a quadric and a quartic surface. By Remark \ref{r2secantr3} and Proposition \ref{assumexsmoothable}, $C^9$ satisfies the assumptions of Theorem \ref{properties} (with $(\star \star)_3$ of Proposition \ref{r2secant} holding only in the restricted range $4 \leq h \leq 5$). In particular, a curve $X^9_4$ is defined. Moreover, $C^9 \cup n$ is a 2-normal curve of arithmetic genus $10$ and degree $9$ in $\mathbb{P}^3$, hence it is smoothable to $C^{10}$ by \cite[IV, Example 6.4.3]{HA} (showing smoothability is an easy exercise).
\item[{$\left[r \geq 4\right]$}] Let $Z=X^{g-1}_h \cup m$ and $X^{g-1}_h \cup l$ be as in Lemma \ref{h0nx} (recall that we are assu-ming $h \leq r+1$). Let $n$ be a 2-secant line to $C^{g-1}$ (note that $C^{g-1}$ always exists). By Lemma \ref{sameirreduciblecomponent}, $C^{g-1} \cup n$ and $C^g$ belong to the same irreducible component of $\text{Hilb}^r_{g,g-1}$. By inductive assumption, the curve $X^{g-1-r}_r \subset \mathbb{P}^r$ is smoothable to $C^{g-1}$. Let $B_h \subset V^{h-1}_{h}(\mathcal{O}_{X^{g-1-r}_r}(1))$ and $V^{C^{g-1}}_h \subset V^{h-1}_{h}(\mathcal{O}_{C^{g-1}}(1))$  be the irreducible components exhibited in the proof of Proposition \ref{r2secant}. Now, $B_h$ and $V^{C^{g-1}}_h$ correspond in an obvious way to irreducible components (call then again $B_h$ and $V^{C^{g-1}}_h$) of $V^{h-1}_{h}(\mathcal{O}_{X^{g-1-r}_r \cup s}(1))$ and $V^{h-1}_{h}(\mathcal{O}_{C^{g-1} \cup n}(1))$, respectively, where $s$ is a general line meeting each of the two components of $X^{g-1-r}_r$ at one point. Let $\mathcal{X}_{(h)}$ be the $h$-symmetric product of the universal family $\mathbb{P}^{r+1} \times W^r_{p_a(X^{g-1-r}_{r+1})} \supset \mathcal{X} \xrightarrow{\Psi} W^r_{p_a(X^{g-1-r}_{r+1})}$, let $\mathcal{L} := \mathcal{O}_{\mathcal{X}}(1) \in \text{Pic}(\mathcal{X})$ be the relative hyperplane bundle and let $\mathcal{V}^{h-1}_h(\mathcal{L})$ be as in the proof of Proposition \ref{r2secant}. By inductive assumption and Lemma \ref{sameirreduciblecomponent}, the curves $X^{g-1-r}_r \cup s$, $X^{g-1-r}_{r+1}$, $C^{g-1} \cup n$ and $C^g$ are all fibres of $\Psi$. It is then clear from the proof of Proposition \ref{r2secant} %(and, for $r=3$, by Remark \ref{r2secantr3})
    that $V^{C^g}_h$ and $V^{C^{g-1}}_h$ are restrictions to the respective fibres of $\Psi$ of an irreducible component of $\mathcal{V}^{h-1}_h(\mathcal{L})$ extending $B_h$, thus the divisor $S^{C^{g-1}}_h$ on $C^{g-1} \cup n$ (which is actually supported on $C^{g-1}$) deforms to the divisor $S^{C^g}_h$ on $C^g$.\\
    %senza prendere in considerazione le V_h, la curva potrebbe deformarsi a una \tilde{X}^g_h, su cui lo schema di intersezione S_h potrebbe essere parametrizzato da una c.irrid. di V^{h-1}_h(\mathcal{O}_{C^g}(1)) che non soddisfa tutte le (**)_r.
    As a consequence, if there exists a partial smoothing of $X^{g-1}_h \cup n$ preserving the nodes at $S^{C^{g-1}}_h$, that curve deforms to $X^g_h$. In order to show that this partial smoothing exists,
    consider then the short exact sequence (see (\ref{icinxnprimex}))
    \begin{equation}
    \label{nxg1hn}
    0 \rightarrow {N_{X^{g-1}_h \cup n}}_{|E_h}(-S_h) \rightarrow N^{S_h}_{N_{X^{g-1}_h \cup n}} \rightarrow N_{C^{g-1} \cup n} \rightarrow 0.
    \end{equation}
    Since $n \cap E_h=\emptyset$, one has $h^1({N_{X^{g-1}_h \cup n}}_{|E_h}(-S_h))=h^1({N_{X^{g-1}_h}}_{|E_h}(-S_h))$, which is $0$ by Theorem \ref{properties}. Sequence (\ref{nxg1hn}) then yields the surjectivity of the map $\alpha: H^0(N^{S_h}_{N_{X^{g-1}_h \cup n}}) \rightarrow H^0(N_{C^{g-1} \cup n})$. A short exact sequence analogous to (\ref{ncmh}) gives $h^0(N_{C^{g-1} \cup n/H})=3g-4+{r+2 \choose 2}$, from which $h^0(N_{C^{g-1} \cup n})=3g-4+{r+2 \choose 2}+r+1$ and thus $h^0(N^{S_h}_{N_{X^{g-1}_h \cup n}})=3(g+h)-4+{r+2 \choose 2}+r+1$. By Proposition \ref{dimwrgh}, the latter number equals $\dim W^{r+1}_{p_a(X^g_h),h}$, which is exactly the number of parameters a deformation of $X^{g-1}_h \cup n$ preserving the nodes at $S_h$ depends on. Since (see Subsection \ref{dec}) $H^0(N^{S_h}_{X^{g-1}_h \cup n})$ is isomorphic to the tangent space at $[X^{g-1}_h \cup n]$ to the locally closed subscheme of $\text{Hilb}^{r+1}_{p_a(X^g_h),p_a(X^g_h)-1}$ which
    %locally in a neighborhood of $[X^{g-1}_h \cup n]$
    parameterizes deformations of $X^{g-1}_h \cup n$ preserving the nodes at $S_h$, that scheme must be smooth at $[X^{g-1}_h \cup n]$. As a consequence, the surjectivity of $\alpha$ gives that, provided $X^{g-1}_h$ is chosen in the suitable irreducible component of $W^{r+1}_{p_a(X^{g-1}_h),h}$ (i.e., notation as in the proof of Proposition \ref{dimwrgh}, $E_h$ is chosen in the suitable irreducible component $R_h$), the curves $X^{g-1}_h \cup n$, $X^g_h$ (and, by construction, $Z$) are parameterized by the same irreducible component of $W^{r+1}_{p_a(X^g_h),h}$.
    %X^{g-1}_h \cup n si deforma a una curva del tipo X^g_h che, scegliendo opportunamente la c.irr. a cui appartiene X^{g-1}_h, è proprio X^g_h.
    %da r+1=5 in su ho sempre a disposizione per induzione, nel caso g={r+2 \choose 2}, una curva C^g-1, cosa che non è ovvia nel caso r+1=4, perché C^9 non viene dall'induzione ma devo mostrare ad hoc che esiste.
\end{itemize}
\end{itemize}
\begin{centering}
\begin{figure}[h]
%\hspace{100pt}
\centering
\includegraphics[width=75mm]{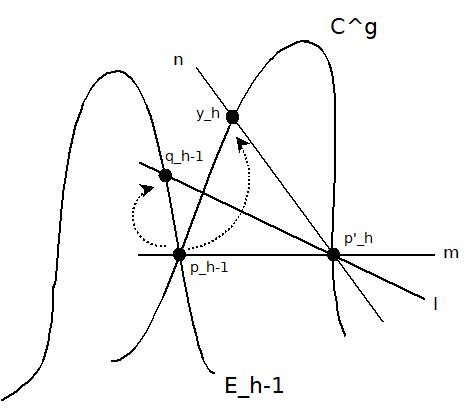}
\caption{if $X^g_h$ admits only locally trivial deformations, $X^g_{h-1} \cup l$ and $X^g_{h-1} \cup n$ are necessarily parameterized by two different irreducible components of $\text{Hilb}^{r+1}_{p_a(X^g_h),p_a(X^g_h)-1}$, whose intersection scheme contains $[Z]$. The dotted arrows repre-sent the two (flat) deformations of $Z$.}
\end{figure}
\end{centering}
%\newline
Now, by \cite[Theorem 1.3.2]{TES}, the curves $\left\{\begin{array}{cc}
                                                 X^g_{h} \hbox{ and } X^g_{h-1} \cup n, & g \leq {r+2 \choose 2} -1 \\
                                                 X^{g-1}_{h} \cup l \hbox{ and } X^{g}_{h}, & g={r+2 \choose 2}
                                               \end{array}\right\}$
%le due curve variano in due famiglie che hanno la stessa dimensione, ma nessuna di esse è specializzazione dell'altra, quindi...
must necessarily be parameterized by two different irreducible components of $\text{Hilb}^{r+1}_{p_a(X^g_h),p_a(X^g_h)-1}$, whose intersection scheme contains $[Z]$ (see Figure 1, representing the case $g \leq {r+2 \choose 2} -1$, the other one is analogous), so that $\dim T_{[Z]} \text{Hilb}^{r+1}_{p_a(X^g_h),p_a(X^g_h)-1}=h^0(N_Z)$ must be greater than $h^0(N_{X^g_h})$, which by Lemma \ref{h0nx} is a contradiction.\\
Then, by Proposition \ref{dimwrgh}, $X^g_h$ must be parameterized by an irreducible component $W^{r+1}_{p_a(X^g_h)} \subset \text{Hilb}^{r+1}_{p_a(X^g_h),p_a(X^g_h)-1}$ such that $\dim W^{r+1}_{p_a(X^g_h)} \geq 3p_a(X^g_h)-4+{r+3 \choose 2}$. Lemma \ref{h0nx} yields equality and the fact that $[X^g_h]$ is a smooth point of $W^{r+1}_{p_a(X^g_h)}$ (and of $\text{Hilb}^{r+1}_{p_a(X^g_h),p_a(X^g_h)-1}$). Now, using Theorem \ref{properties} and Proposition \ref{hilbertscheme} to compute cohomology of the exact sequence (see (\ref{icinxnprimex}))
$$0 \rightarrow {N_{X^g_h}}_{|E_h}(-S_h) \rightarrow N^{\prime}_{X^g_h} \rightarrow N_{C^g} \rightarrow 0$$
%qui c'è implicito che h^0(N_C/P^r+1)=h^0(N_C/P^r)+r+1.
one obtains $h^0(N^{\prime}_{X^g_h})=3p_a(X^g_h)-5+{r+3 \choose 2} < h^0(N_{X^g_h})$, hence Lemma \ref{critical} applies and $X^g_h$ is smoothable. Proposition \ref{assumexsmoothable} then gives that $X^g_h$ is smoothable to a half-canonical curve $\Gamma=\Gamma_{r+1,g+h} \subset \mathbb{P}^{r+1}$ satisfying properties $(*)$. In particular, the Gaussian map $\Psi_{\mathcal{O}_{\Gamma}(1)}$ is injective, hence, by Corollary \ref{cornag}, the pair $(\Gamma,\mathcal{O}_{\Gamma}(1))$ is parameterized by an irreducible component $\widetilde{V}_{r+1,g+h} \subset \mathcal{S}^{r+1}_{g+h}$ having expected codimension in $\mathcal{S}_{g+h}$. Moreover, by Proposition \ref{r2secant}, $\Gamma_{r+1,g+h}$ satisfies $(\star \star)_{r+1}$. Then one defines $C^{g+h}:=\Gamma_{r+1,g+h} \subset \mathbb{P}^{r+1}$ and the inductive step is proved.
%and $[\Gamma]$ is a general point of an irreducible component $\text{Hilb}^{r+1}_{g(\Gamma),g(\Gamma)-1}$ of dimension greater than $(r+1)^2+(r+1)-2$ (see Proposition \ref{hilbertscheme}), assunzione decaduta in seguito alla modifica della Proposition 2.16.
\end{proof}
Proposition \ref{farkasinductive} (induction on $g$) applied on the components $\widetilde{V}_{r,g(r)} \subset \mathcal{S}^r_{g(r)}$ immediately gives the main result of the paper:
\begin{teo}
For all integers $r \geq 2$ and $g \geq g(r) := \lfloor \frac{r^2+10r+1}{4} \rfloor$, there exists an irreducible component $V_{r,g} \subset \mathcal{S}^r_{g}$ having expected codimension in $\mathcal{S}_{g}$. Moreover, the general point $(C,L) \in V_{r,g}$ has very ample $L$.
\end{teo}
\begin{oss}
Taking the value $h(r) := \lfloor \frac{r}{2} \rfloor +3$, at every step of the induction, as the number of intersection points of the ``base" half-canonical curve $\Gamma_{r,g(r)} \subset \mathbb{P}^{r}$ with the elliptic normal curve, we find, by Theorem \ref{maineq}, smooth half-canonical curves of genus $g(r+1)$ in $\mathbb{P}^{r+1}$ having the desired properties and attaining the minimum possible increment in the genus with respect to $r$.\\
On the other hand, Theorem \ref{maineq} yields as well the existence of smooth half-canonical curves in $\mathbb{P}^r$ of genus $g(r) < g \leq {r+2 \choose 2}$, which correspond to smooth points of irreducible components $\widetilde{V}_{r,g} \subset \mathcal{S}^{r}_g$ having expected codimension in $\mathcal{S}_g$. It would be interesting to understand whether, for these values of $g$, one has $V_{r,g}=\widetilde{V}_{r,g}$.
\end{oss}
%\begin{oss}
%The value $h(r) := \lfloor \frac{r}{2} \rfloor +2$, which, at every step on the induction, is the number of intersection points of the ``base" half-canonical curve $\Gamma_{r,g(r)}$ with the elliptic normal curve, is the smallest number $h$ for which there exist elliptic normal curves in $\mathbb{P}^{h-1}$ intersecting $\Gamma_{r,g(r)}$ in $h+2$ points (cf. Proposition \ref{r2secant}). Of course, we have used this value in the proof of Theorem \ref{maineq} in order to attain the minimum possible increment in the genus $g(r)$ of $\Gamma_{r,g(r)}$ at every step of the induction.
%Nevertheless, the same argument used in the proof of Theorem \ref{maineq} gives, for every $r \geq 2$, the existence of half-canonical curves $\Gamma \subset \mathbb{P}^r$ of genus $g(\Gamma) > g(r)$. It is sufficient to attach to the ``base" half-canonical curve an elliptic normal curve intersecting it in a number of points greater than $\lfloor \frac{r}{2} \rfloor +2$ (cf. Lemma \ref{nye}). It would be interesting to understand whether the points $(\Gamma,\mathcal{O}_{\Gamma}(1))$ belong to the same irreducible component whose existence is guaranteed by Theorem \ref{farkasinductive} once Theorem \ref{maineq} is proved. Note that, since we start induction on $g$ from spin curves with a very ample theta characteristics, Remark 2.5 of \cite{FA} ensures that the general point of these components parameterizes a spin curve with very ample theta too.
%\end{oss}
\end{section}
$\left.\right.$\\ \\
\textbf{Acknowledgements.}\\
This research project was partially supported by PRIN 2010-2011 ``Geometry of algebraic varieties" and by FIRB 2012 ``Moduli spaces and Applications".\\
The author thanks Edoardo Ballico, Claudio Fontanari, Trygve Johnsen, Roberto Pignatelli and Edoardo Sernesi for helpful discussions.

\vspace{0.3cm}

\noindent
Luca Benzo \newline
Dipartimento di Scienze Matematiche ``Giuseppe Luigi Lagrange" \newline
Politecnico di Torino \newline
Corso Duca degli Abruzzi 24 \newline
10129 Torino, Italy. \newline
E-mail address: luca.benzo@polito.it

\small
\begin{thebibliography}{plain}

\bibitem{ACG} E. Arbarello, M. Cornalba, P. Griffiths, \emph{Geometry of algebraic curves vol. II}, Springer-Verlag, 2011.

\bibitem{BABEF} E. Ballico, L. Benzo, C. Fontanari, \emph{Families of nodal curves in $\mathbb{P}^r$ with the expected number of moduli}, Bollettino dell'Unione Matematica Italiana 7-3 (2014), pp. 183-192.

\bibitem{BEA} A. Beauville, \emph{Prym varieties and the Schottky problem}, Invent. Math. 41 (1977), 149--196.

\bibitem{BE} L. Benzo, \emph{Components of moduli spaces of spin curves with the expected codimension}, Math.
Ann. 363-1/2 (2015), 385--392.

\bibitem{CI} C. Ciliberto, \emph{On the Hilbert scheme of curves of maximal genus in a projective space}, Math. Z. 194 (1987), 351--363.

\bibitem{CM} M. Coppens, G. Martens, \emph{Secant spaces and Clifford's theorem}, Compositio Math., 78-2 (1991), 193--212.

\bibitem{COR} M. Cornalba, \emph{Moduli of curves and theta-characterstics}, in Lectures on Riemann surfaces, Trieste
(1987), 560--589.

\bibitem{EL} L. Ein, R. Lazarsfeld, \emph{Stability and restrictions of Picard bundles, with an application to the normal bundles of elliptic curves}, in Complex Projective Geometry: selected papers, London Mathematical Society Lecture Note Series 179, Cambridge University Press, 1992.

\bibitem{FA} G. Farkas, \emph{Gaussian maps, Gieseker-Petri loci and large theta-characteristics}, J. reine angew. Math. 581 (2005), 151--173.

\bibitem{FP} G. Farkas, M. Popa, \emph{Effective divisors on $\overline{\mathcal{M}}_g$, curves on $K3$ surfaces, and the slope conjecture},
J. Alg. Geom. 14 (2005), 241--267.

\bibitem{HA} J. Harris, \emph{Theta-characteristics on algebraic curves}, Trans. Amer. Math. Soc. 271 (1982), 611--638.

\bibitem{HART}, R. Hartshorne, \emph{Algebraic geometry}, Springer Verlag, 1977.

\bibitem{HJ} M.E. Huibregtse, T. Johnsen, \emph{Local properties of secant varieties in symmetric products. Part I}, Trans. Amer. Math. Soc. 313-1 (1989), 187--204.

%\bibitem{HH} R. Hartshorne, A. Hischowitz, \emph{Smoothing algebraic space curves}, Algebraic Geometry Sitges (Barcelona) 1983,
%Lecture Notes in Mathematics 1124 (1985), Springer, 98--131.

%\bibitem{HH} R. Hartshorne and A. Hirschowitz, \emph{Smoothing algebraic space curves}, In: Algebraic Geometry, Proceedings Sitges 1983, Lecture Notes in Mathematics, Vol. 1124, 98--131, Springer, 1985.

\bibitem{KL} J.O. Kleppe, \emph{Liaison of families of subschemes in $\mathbb{P}^n$}, in Algebraic Curves and Projective Geometry, Lec. Notes in Math. 1389, Springer-Verlag, 1989, 128--173.

\bibitem{KO} J. Koll\'ar, \emph{Rational curves on algebraic varieties}, Springer-Verlag, 1999.

\bibitem{MA} A. Mattuck, \emph{Secant bundles on symmetric products}, Amer. J. Math. 87 (1965), 779--797.

\bibitem{MU} D. Mumford, \emph{Theta-characteristics on algebraic curves}, Ann. Ecole Norm. Sup. 4-4 (1971), 181--192.

\bibitem{NAG} D. S. Nagaraj, \emph{On the moduli space of curves with theta-characteristics}, Compos. Math. 75 (1990), 287--297.

%\bibitem{S} E. Sernesi, \emph{Deformations of algebraic schemes}, Springer-Verlag, 2006.

%\bibitem{RA} Z. Ran, \emph{Normal bundles of rational curves in projective spaces}, Asian J. Math. 11-4 (2007), 567--608.
%Era citato nel Lemma 2.10 per la definizione di fibrato normale relativo, ma rileggendo meglio quella definizione si applica male al mio caso.

\bibitem{SA} I. Sabadini, \emph{On the Hilbert scheme of curves of degree $d$ and genus ${d-3 \choose 2}-1$}, Le Matematiche 55-2 (2000), 517--531.

\bibitem{SCHW} R.L.E. Schwarzenberger, \emph{The secant bundle of a projective variety}, Proceedings of the L.M.S. 14 (1964), 369--384.

\bibitem{Sernesi} E. Sernesi, \emph{On the existence of certain families of curves},
Invent. Math. 75 (1984), 25--57.

\bibitem{SE} E. Sernesi, \emph{Deformations of algebraic schemes}, Springer-Verlag, 2006.

\bibitem{TEI} M. Teixidor, \emph{Half-canonical series on algebraic curves}, Trans. Amer. Math. Soc. 302-1 (1987), 99--115.

\bibitem{TES} B. Tessier, \emph{R\'esolution simultan\'ee: I - Familles de courbes},  S\'eminaire sur les singularit\'es des surfaces no. 8 (1976-1977), 1--10.

\bibitem{WA} J. Wahl, \emph{Introduction to Gaussian maps on an algebraic curve}, in Complex Projective Geometry, London
Math. Soc. Lect. Notes Ser. 179 (1992), 304--323.

\end{thebibliography}
\end{document}